\documentclass[a4paper, 10pt, conference]{ieeeconf}
\IEEEoverridecommandlockouts 
\usepackage{amsmath,amssymb,url}
\usepackage{graphicx,subfigure,warmread}
\usepackage[all,import]{xy}
\usepackage{color}

\newcommand{\norm}[1]{\ensuremath{\left\| #1 \right\|}}

\newcommand{\bracket}[1]{\ensuremath{\left[ #1 \right]}}
\newcommand{\braces}[1]{\ensuremath{\left\{ #1 \right\}}}

\newcommand{\refeqn}[1]{(\ref{eqn:#1})}
\newcommand{\reffig}[1]{Figure \ref{fig:#1}}
\newcommand{\tr}[1]{\mbox{tr}\ensuremath{\negthickspace\bracket{#1}}}
\newcommand{\trs}[1]{\mathrm{tr}\ensuremath{[#1]}}

\newcommand{\SO}{\ensuremath{\mathsf{SO(3)}}}
\newcommand{\T}{\ensuremath{\mathsf{T}}}
\renewcommand{\L}{\ensuremath{\mathsf{L}}}
\newcommand{\so}{\ensuremath{\mathfrak{so}(3)}}
\newcommand{\SE}{\ensuremath{\mathsf{SE(3)}}}

\renewcommand{\Re}{\ensuremath{\mathbb{R}}}
\newcommand{\Sph}{\ensuremath{\mathsf{S}}}

\newcommand{\D}{\ensuremath{\mathbf{D}}}

\title{\LARGE \bf
Geometric Nonlinear PID Control of a Quadrotor UAV on \SE}

\author{Farhad Goodarzi, Daewon Lee, and Taeyoung Lee\authorrefmark{1}%, Melvin Leok\authorrefmark{2}, and N. Harris McClamroch%
\thanks{Farhad Goodarzi, Daewon Lee, Taeyoung Lee, Mechanical and Aerospace Engineering, The George Washington University, Washington DC 20052 {\tt \{fgoodarzi,daewonlee,tylee\}@gwu.edu}}%
\thanks{\textsuperscript{\footnotesize\ensuremath{*}}This research has been supported in part by NSF under the grant CMMI-1243000 (transferred from 1029551).}
}

\newtheorem{prop}{Proposition}

\begin{document}
\allowdisplaybreaks
\maketitle \thispagestyle{empty} \pagestyle{empty}

\begin{abstract}
Nonlinear PID control systems for a quadrotor UAV are proposed to follow an attitude tracking command and a position tracking command. The control systems are developed directly on the special Euclidean group to avoid singularities of minimal attitude representations or ambiguity of quaternions. A new form of integral control terms is proposed to guarantee almost global asymptotic stability when there exist uncertainties in the quadrotor dynamics. A rigorous mathematical proof is given. Numerical example illustrating a complex maneuver, and a preliminary experimental result are provided.

\end{abstract}

\section{INTRODUCTION}

%A quadrotor unmanned aerial vehicle (UAV) consists of two pairs of counter-rotating rotors and propellers. It has been envisaged for various applications such as surveillance or sensor networks as well as for educational purposes.

%Linear control systems have been widely used to enhance the stability properties of an equilibrium of a quadrotor UAV~\cite{ValBetPAGNCC06,HofHuaAGNCC07,CasLozICSM05}. In~\cite{GilHofIJRR11}, the quadrotor dynamics is modeled as a collection of simplified hybrid dynamic modes, and  reachability sets are analyzed to guarantees the safety and performance for larger area of operating conditions.

A quadrotor unmanned aerial vehicle (UAV) has been envisaged for various applications such as surveillance, sensing or educational purposes, due to its ability to hover with simpler mechanical structures compared to helicopters. Several control systems have been developed based on backstepping, sliding mode controller, or adaptive neural network~\cite{BouSiePIICRA05,EfePMCCA07,NicMacPCCECE08}. Aggressive maneuvers are also demonstrated at \cite{MelMicIJRR12}. However, these are based on Euler angles. Therefore they involve complicated expressions for trigonometric functions, and they exhibit singularities which restrict their ability to achieve complex rotational maneuvers significantly.

There are quadrotor control systems developed in terms of quaternions~\cite{TayMcGITCSTI06}. Quaternions do not have singularities but, as the three-sphere double-covers the special orthogonal group, one attitude may be represented by two antipodal points on the three-sphere. This ambiguity should be carefully resolved in quaternion-based attitude control systems, otherwise they may exhibit unwinding, where a rigid body unnecessarily rotates through a large angle even if the initial attitude error is small~\cite{BhaBerSCL00}. To avoid these, an additional mechanism to lift attitude onto the unit-quaternion space is introduced~\cite{MaySanITAC11}.

There are other limitations of quadrotor control systems such as complexities in controller structures or lack of stability proof. For example, tracking control of a quadrotor UAV has been considered in~\cite{CabCunPICDC09,MelKumPICRA11}, but the control system in~\cite{CabCunPICDC09} has a complex structure since it is based on a multiple-loop backstepping approach, and no stability proof is presented in~\cite{MelKumPICRA11}. Robust tracking control systems are studied in~\cite{NalMarPICDC09,HuaHamITAC09}, but the quadrotor dynamics is simplified by considering planar motion only~\cite{NalMarPICDC09}, or by ignoring the rotational dynamics by timescale separation assumption~\cite{HuaHamITAC09}. 

%, but it is based on Euler-angles. 

Recently, the dynamics of a quadrotor UAV is globally expressed on the special Euclidean group, $\SE$, and nonlinear control systems are developed to track outputs of several flight modes~\cite{LeeLeoPICDC10}. Several aggressive maneuvers of a quadrotor UAV are demonstrated based on a hybrid control architecture, and a nonlinear robust control system is also considered in~\cite{LeeLeoPACC12}. As they are directly developed on the special Euclidean group, complexities, singularities, and ambiguities associated with minimal attitude representations or quaternions are completely avoided~\cite{ChaSanICSM11}.

This paper is an extension of the prior works of the author in~\cite{LeeLeoPICDC10,LeeLeoPACC12}. It is assumed that there exist uncertainties on the translational dynamics and the rotational dynamics of a quadrotor UAV, and nonlinear PID controllers are proposed to follow an attitude tracking command and a position tracking command. Linear or nonlinear PID controllers have been widely used in various experimental settings for a quadrotor UAV, without careful stability analyses. This paper provides a new form of integral control terms that guarantees asymptotic convergence of tracking errors with uncertainties. The nonlinear robust tracking control system in~\cite{LeeLeoPACC12} provides ultimate boundedness of tracking errors, and the control input may be prone to chattering if the required ultimate bound is smaller. Compared with~\cite{LeeLeoPACC12}, the control system in this paper provides stronger asymptotic stability, and there is no concern for discontinuities. The structure of the control system is also simplified such that the cross term of the angular velocity does not have to be cancelled.

%In this paper, we extend the results of \cite{LeeLeo} to construct nonlinear robust tracking control systems on $\SE$ for a quadrotor UAV. We assume that there exist unstructured, bounded uncertainties, with pre-determined bounds, on the translational dynamics and the rotation dynamics of a quadrotor UAV. Output tracking control systems are developed to follow an attitude command or a position command for the vehicle center of mass. We show that the tracking errors are uniformly ultimately bounded, and the size of the ultimate bound can be arbitrarily reduced. The robustness of the proposed tracking control systems are critical in generating complex maneuvers, as the impact of the several aerodynamic effects resulting from the variation in air speed is significant even at moderate velocities~\cite{HofHuaAGNCC07}.

In short, the unique features of the control system proposed in this paper are as follows: (i) it is developed for the full six degrees of freedom dynamic model of a quadrotor UAV on $\SE$, including the coupling between the translational dynamics and the rotational dynamics, (ii) a rigorous Lyapunov analysis is presented to establish stability properties without any timescale separation assumption, and (iii) it is guaranteed to be robust against unstructured uncertainties in both the translational dynamics and the rotational dynamics, (iv) in contrast to hybrid control systems~\cite{GilHofIJRR11}, complicated reachability set analysis is not required to guarantee safe switching between different flight modes, as the region of attraction for each flight mode covers the configuration space almost globally. To the author's best knowledge, a rigorous mathematical analysis of nonlinear PID-like controllers of a quadrotor UAV with almost global asymptotic stability on $\SE$ has been unprecedented.

\section{QUADROTOR DYNAMICS MODEL}\label{sec:QDM}

Consider a quadrotor UAV model illustrated in \reffig{QM}. %This is a system of four identical rotors and propellers located at the vertices of a square, which generate a thrust and torque normal to the plane of this square. 
We choose an inertial reference frame $\{\vec e_1,\vec e_2,\vec e_3\}$ and a body-fixed frame $\{\vec b_1,\vec b_2,\vec b_3\}$. The origin of the body-fixed frame is located at the center of mass of this vehicle. The first and the second axes of the body-fixed frame, $\vec b_1,\vec b_2$, lie in the plane defined by the centers of the four rotors.%, as illustrated in \reffig{QM}. %The third body-fixed axis $\vec b_3$ is normal to this plane.  %Each of the inertial reference frame and the body-fixed reference frame consist of a triad of orthogonal vectors defined according to the right hand rule.    

The configuration of this quadrotor UAV is defined by the location of the center of mass and the attitude with respect to the inertial frame. Therefore, the configuration manifold is the special Euclidean group $\SE$, which is the semidirect product of $\Re^3$ and the special orthogonal group $\SO=\{R\in\Re^{3\times 3}\,|\, R^TR=I,\, \det{R}=1\}$.

The mass and the inertial matrix of a quadrotor UAV are denoted by $m\in\Re$ and $J\in\Re^{3\times 3}$. Its attitude, angular velocity, position, and velocity are defined by $R\in\SO$, $\Omega,x,v\in\Re^3$, respectively, where the rotation matrix $R$ represents the linear transformation of a vector from the body-fixed frame to the inertial frame and the angular velocity $\Omega$ is represented with respect to the body-fixed frame. The distance between the center of mass to the center of each rotor is $d\in\Re$, and the $i$-th rotor generates a thrust $f_i$ and a reaction torque $\tau_i$ along $-\vec b_3$ for $1\leq i \leq 4$. The magnitude of the total thrust and the total moment in the body-fixed frame are denoted by $f, M\in\Re^3$, respectively. 

%Define
%\begin{center}
%\begin{tabular}{lp{5.6cm}}
%{$m\in\Re$} & the total mass\\
%{$J\in\Re^{3\times 3}$} & the inertia matrix with respect to the body-fixed frame\\
%{$R\in\SO$} & the rotation matrix from the body-fixed frame to the inertial  frame\\
%{$\Omega\in\Re^3$} & the angular velocity in the body-fixed frame\\
%{$x\in\Re^3$} & the position vector of the center of mass in the inertial frame\\
%{$v\in\Re^3$} & the velocity vector of the center of mass in the inertial frame\\
%{$d\in\Re$} & the distance from the center of mass to the center of each rotor in the $\vec b_1,\vec b_2$ plane\\
%{$f_i\in\Re$} & the thrust generated by the $i$-th propeller along the $-\vec b_3$ axis\\
%{$\tau_i\in\Re$} & the torque generated by the $i$-th propeller about the $\vec b_3$ axis\\
%{$f\in\Re$} & the total thrust magnitude, i.e., $f=\sum_{i=1}^4 f_i$\\
%{$M\in\Re^3$} & the total moment vector in the body-fixed frame\\
%\end{tabular}
%\end{center}

\begin{figure}
\setlength{\unitlength}{0.65\columnwidth}\footnotesize
\centerline{
\begin{picture}(1,0.8)(0,0)
\put(0,0){\includegraphics[width=0.65\columnwidth]{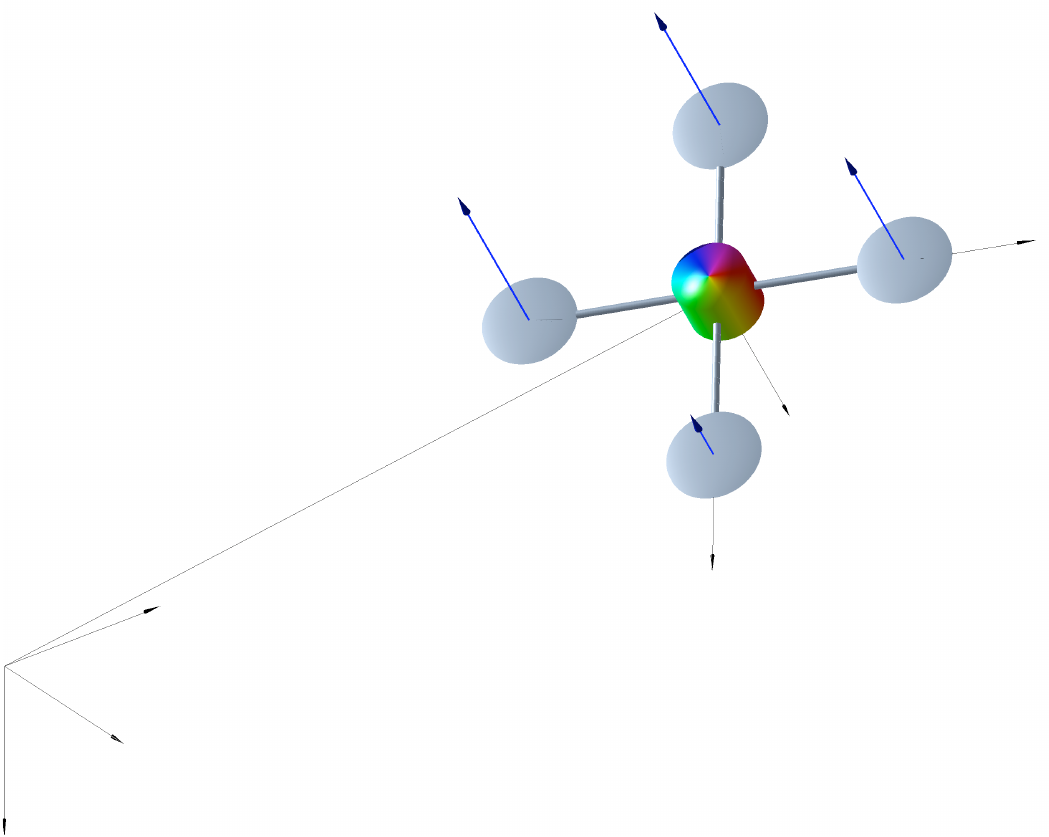}}
\put(0.16,0.18){\shortstack[c]{$\vec e_1$}}
\put(0.13,0.06){\shortstack[c]{$\vec e_2$}}
\put(0.02,0.0){\shortstack[c]{$\vec e_3$}}
\put(0.98,0.5){\shortstack[c]{$\vec b_1$}}
\put(0.70,0.22){\shortstack[c]{$\vec b_2$}}
\put(0.76,0.37){\shortstack[c]{$\vec b_3$}}
\put(0.78,0.66){\shortstack[c]{$f_1$}}
\put(0.56,0.76){\shortstack[c]{$f_2$}}
\put(0.40,0.63){\shortstack[c]{$f_3$}}
\put(0.61,0.42){\shortstack[c]{$f_4$}}
\put(0.30,0.35){\shortstack[c]{$x$}}
\put(0.90,0.35){\shortstack[c]{$R$}}
\end{picture}}
\caption{Quadrotor model}\label{fig:QM}
\end{figure}
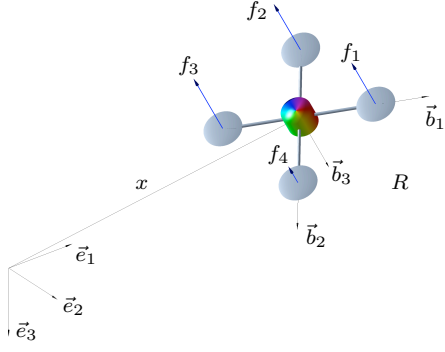

The following conventions are assumed for the rotors and propellers, and the thrust and moment that they exert on the quadrotor UAV. We assume that the thrust of each propeller is directly controlled, and the direction of the thrust of each propeller is normal to the quadrotor plane. The first and third propellers are assumed to generate a thrust along the direction of $-\vec b_3$ when rotating clockwise; the second and fourth propellers are assumed to generate a thrust along the same direction of $-\vec b_3$ when rotating counterclockwise. Thus, the thrust magnitude is $f=\sum_{i=1}^4 f_i$, and it is positive when the total thrust vector acts  along $-\vec b_3$, and it is negative when the total thrust vector acts along $\vec b_3$. By the definition of the rotation matrix $R\in\SO$, 
%the direction of the $i$-th body-fixed axis $\vec b_i$ is given by $Re_i$ in the inertial frame, where $e_1=[1;0;0],e_2=[0;1;0],e_3=[0;0;1]\in\Re^3$. Therefore, 
the total thrust vector is given by $-fRe_3\in\Re^3$ in the inertial frame.
We also assume that the torque generated by each propeller is directly proportional to its thrust. Since it is assumed that the first and the third propellers rotate clockwise and the second and the fourth propellers rotate counterclockwise to generate a positive thrust along the direction of $-\vec b_3$, %
the torque generated by the $i$-th propeller about $\vec b_3$ can be written as $\tau_i=(-1)^{i} c_{\tau f} f_i$  for a fixed constant $c_{\tau f}$.    All of these assumptions are fairly common in many quadrotor control systems~\cite{TayMcGITCSTI06,CasLozICSM05}.  %The presented control system can readily be extended to include linear rotor dynamics, as studied in~\cite{BouSiePIICRA05}.

Under these assumptions, the thrust of each propeller $f_1, f_2, f_3, f_4$ is directly converted into $f$ and $M$, or vice versa. In this paper, the thrust magnitude $f\in\Re$ and the moment vector $M\in\Re^3$ are viewed as control inputs. 
%%the moment vector in the body-fixed frame is given by 
%\begin{align}
%\begin{bmatrix} f \\ M_1 \\ M_2 \\ M_3 $ $\end{bmatrix}
%=\begin{bmatrix} 1 & 1& 1& 1\\
%0 & -d & 0 & d\\
%d & 0 & -d & 0\\
%-c_{\tau f} & c_{\tau f} & -c_{\tau f} &c_{\tau f}
%\end{bmatrix}
%\begin{bmatrix} f_1 \\ f_2\\ f_3 \\ f_4\end{bmatrix}\label{eqn:fM}.
%\end{align}
%The determinant of the above $4\times4$ matrix is $8 c_{\tau f} d^2$, so it is invertible when $d\neq 0$ and $c_{\tau f}\neq 0$. Therefore, for given thrust magnitude $f$ and given moment vector $M$, the thrust of each propeller $f_1, f_2, f_3, f_4$ can be obtained from \refeqn{fM}. Using this equation, the thrust magnitude $f\in\Re$ and the moment vector $M\in\Re^3$ are viewed as control inputs in this paper.
The equations of motion are given by
\begin{gather}
\dot x  = v,\label{eqn:EL1}\\
m \dot v = mge_3 - f R e_3 + \Delta_x,\label{eqn:EL2}\\
\dot R = R\hat\Omega,\label{eqn:EL3}\\
J\dot \Omega + \Omega\times J\Omega = M + \Delta_R,\label{eqn:EL4}
\end{gather}
where the \textit{hat map} $\hat\cdot:\Re^3\rightarrow\so$ is defined by the condition that $\hat x y=x\times y$ for all $x,y\in\Re^3$. 
%More explicitly, for a vector $x=[x_1;x_2;x_3]\in\Re^3$, the matrix $\hat x$ is given by
%\begin{align}
%    \hat x = \begin{bmatrix} 0 & -x_3 & x_2\\
%                                x_3 & 0 & -x_1\\
%                                -x_2 & x_1 & 0 \end{bmatrix}\label{eqn:hat}.
%\end{align}
This identifies the Lie algebra $\so$ with $\Re^3$ using the vector cross product in $\Re^3$. %(see Appendix \ref{app:hat}). 
The inverse of the hat map is denoted by the \textit{vee} map, $\vee:\so\rightarrow\Re^3$. Unstructured, but fixed uncertainties in the translational dynamics and the rotational dynamics of a quadrotor UAV are denoted by $\Delta_x$ and $\Delta_R\in\Re^3$, respectively. 

Throughout this paper, $\lambda_m (A)$ and $\lambda_{M}(A)$ denote the minimum eigenvalue and the maximum eigenvalue of a square matrix $A$, respectively, and $\lambda_m$ and $\lambda_M$ are shorthand for $\lambda_m=\lambda_m(J)$ and $\lambda_M=\lambda_M(J)$. The two-norm of a matrix $A$ is denoted by $\|A\|$.

\section{ATTITUDE CONTROLLED FLIGHT MODE}\label{sec:ACFM}

%\subsection{Flight Modes}

Since the quadrotor UAV has four inputs, it is possible to achieve asymptotic output tracking for at most four quadrotor UAV outputs.    The quadrotor UAV has three translational and three rotational degrees of freedom; it is not possible to achieve asymptotic output tracking of both attitude and position of the quadrotor UAV. This  motivates us to introduce two flight modes, namely (1) an attitude controlled flight mode, and (2) a position controlled flight mode. While a quadrotor UAV is underactuated, a complex flight maneuver can be defined by specifying a concatenation of flight modes together with conditions for switching between them. This will be further illustrated by a numerical example later. In this section, an attitude controlled flight mode is considered. 

%A complex flight maneuver can be defined by specifying a concatenation of flight modes together with conditions for switching between them; for each flight mode one also specifies the desired or commanded outputs as functions of time. Unlike a hybrid flight control system that requires reachability analyses~\cite{GilHofIJRR11}, the proposed control system is robust to switching conditions since each flight mode has almost global stability properties, and it is straightforward to design a complex maneuver of a quadrotor UAV. 

%, where the outputs are the attitude of the quadrotor UAV and the controller for this flight mode achieves asymptotic attitude tracking.

\subsection{Attitude Tracking Errors}

Suppose that an  smooth attitude command $R_d(t)\in\SO$ satisfying the following kinematic equation is given:
\begin{align}
\dot R_d = R_d \hat\Omega_d,
\end{align}
where $\Omega_d(t)$ is the desired angular velocity, which is assumed to be uniformly bounded. We first define errors associated with the attitude dynamics as follows~\cite{BulLew05,Lee11}.

\begin{prop}\label{prop:1}
For a given tracking command $(R_d,\Omega_d)$, and the current attitude and angular velocity $(R,\Omega)$, we define an attitude error function $\Psi:\SO\times\SO\rightarrow\Re$, an attitude error vector $e_R\in\Re^3$, and an angular velocity error vector $e_\Omega\in \Re^3$ as follows:
\begin{gather}
\Psi (R,R_d) = \frac{1}{2}\tr{I-R_d^TR},\\
e_R =\frac{1}{2} (R_d^TR-R^TR_d)^\vee,\\
e_\Omega = \Omega - R^T R_d\Omega_d,
\end{gather}
Then, the following properties hold:
\begin{itemize}
\item[(i)] $\Psi$ is positive-definite about $R=R_d$.
\item[(ii)] The left-trivialized derivative of $\Psi$ is given by
\begin{align}
\T^*_I \L_R\, (\D_R\Psi(R,R_d))= e_R.
\end{align}
\item[(iii)] The critical points of $\Psi$, where $e_R=0$, are $\{R_d\}\cup\{R_d\exp (\pi \hat s),\,s\in\Sph^2 \}$.
\item[(iv)] A lower bound of $\Psi$ is given as follows:
\begin{align}
\frac{1}{2}\|e_R\|^2 \leq \Psi(R,R_d),\label{eqn:PsiLB}
\end{align}
\item[(v)] Let $\psi$ be a positive constant that is strictly less than $2$. If $\Psi(R,R_d)< \psi<2$, then an upper bound of $\Psi$ is given by
\begin{align}
\Psi(R,R_d)\leq \frac{1}{2-\psi} \|e_R\|^2.\label{eqn:PsiUB}
\end{align}
\item[(vi)] The time-derivative of $\Psi$ and $e_R$ satisfies:
\begin{align}
\dot\Psi = e_R\cdot e_\Omega,\quad \|\dot e_R\|\leq\|e_\Omega\|.\label{eqn:Psidot00}
\end{align}
\end{itemize}
\end{prop}
\begin{proof} See \cite{Lee11}. \end{proof}

\subsection{Attitude Tracking Controller}

We now introduce a nonlinear controller for the attitude controlled flight mode:%, described by an expression for the moment vector:
\begin{align}
M & = -k_R e_R -k_\Omega e_\Omega -k_I e_I\nonumber\\
&\qquad +(R^TR_d\Omega_d)^\wedge J R^T R_d \Omega_d + J R^T R_d\dot\Omega_d,\label{eqn:aM}\\
e_I & = \int_0^t e_\Omega(\tau)+ c_2e_R(\tau)d\tau,\label{eqn:eI}
\end{align}
where $k_R,k_\Omega,k_I,c_2$ are positive constants. The control moment is composed of proportional, derivative, and integral terms, augmented with additional terms to cancel out the angular acceleration caused by the desired angular velocity. One noticeable difference from the attitude control systems in~\cite{LeeLeoPICDC10,LeeLeoPACC12} is that the cross term at \refeqn{EL4}, namely $\Omega\times J\Omega$ does not have to be cancelled. This simplifies controller structures. 

Unlike common integral control terms where the attitude error is integrated only, here the angular velocity error is also integrated at \refeqn{eI}. This unique term is required to show exponential stability in the presence of the disturbance $\Delta_R$ in the subsequent analysis. From \refeqn{Psidot00}, it essentially increases the proportional term. The corresponding effective controller gains for the proportional term and the integral term are given by $k_R+k_I$ and $c_2k_I$, respectively. We now state the result that the zero equilibrium of tracking errors $(e_R, e_\Omega)$ is exponentially stable.

%In this attitude controlled mode, it is possible to ignore the translational motion of the quadrotor UAV; consequently the reduced model for the attitude dynamics are given by equations \refeqn{EL3}, \refeqn{EL4}, using the controller expression \refeqn{aM}-\refeqn{eI}.  

\begin{prop}{(Attitude Controlled Flight Mode)}\label{prop:Att}
%Suppose that the initial attitude error satisfies
%\begin{gather}
%\Psi(R(0),R_d(0))<\psi_2<2\label{eqn:eRb0}%\\
%%
%\end{gather}
%for a constant $\psi_2$. 
%
Consider the control moment $M$ defined in \refeqn{aM}-\refeqn{eI}. For positive constants $k_R,k_\Omega$, the constants $c_2,B_2$ are chosen such that
\begin{gather}
\|(2J-\trs{J}I)\| \|\Omega_d\| \leq B_2,\label{eqn:B_2}\\
c_2 < \min\bigg\{  \frac{\sqrt{k_R\lambda_m}}{\lambda_M}, \frac{4k_\Omega}{8k_R\lambda_M+(k_\Omega+B_2)^2}\bigg\},\label{eqn:c2}
\end{gather}
Then, the equilibrium of the zero attitude tracking errors $(e_R,e_\Omega,e_I)=(0,0,\frac{\Delta_R}{k_I})$ is almost globally asymptotically stable with respect to $e_R$ and $e_\Omega$\footnote{see \cite[Chapter 4]{HadChe08} for the definition of partial stability}, and the integral term $e_I$ is globally uniformly bounded. It is also locally exponentially stable with respect to $e_R$ and $e_\Omega$.
\end{prop}

\begin{proof}
See Appendix \ref{sec:pfAtt}.
%See \cite{LeeLeo_4457}.
\end{proof}

%From \refeqn{eRb0}, the initial attitude error should be less than $180^\circ$, in terms of the rotation angle about the eigenaxis between $R$ and $R_d$. We can further show that the attitude tracking errors exponentially converges to \refeqn{uubA}, where the size of the ultimate bound can be reduced by the controller parameter $\epsilon_R$. It is also possible to achieve exponential attractiveness if the constant $\epsilon_R$ in \refeqn{muR} is replaced by $\epsilon_R \exp(-\beta t)$ for $\beta >0$. All of these results can be applied to a nonlinear robust control problem for the attitude dynamics of any rigid body.

While these results are developed for the attitude dynamics of a quadrotor UAV, they can be applied to the attitude dynamics of any rigid body. Nonlinear PID-like controllers have been developed for attitude stabilization in terms of modified Rodriguez parameters~\cite{SubJAS04} and quaternions~\cite{SubAkeJGCD04}, and for attitude tracking in terms of Euler-angles~\cite{ShoJuaPACC02}. The proposed tracking control system is developed on $\SO$, therefore it avoids singularities of Euler-angles and Rodriguez parameters, as well as unwinding of quaternions. %It also provides almost global asymptotic stability for attitude \textit{tracking} problems with fixed uncertainties.

Asymptotic tracking of the quadrotor attitude does not require specification of the thrust magnitude.   As an auxiliary problem, the thrust magnitude can be chosen in many different ways to achieve an additional translational motion objective. For example, it can be used to asymptotically track a quadrotor altitude command~\cite{LeeLeo}. Since the translational motion of the quadrotor UAV can only be partially controlled; this flight mode is most suitable for short time periods where an attitude maneuver is to be completed. 

%The translational equations of motion of the quadrotor UAV, during an attitude flight mode, are given by equations \refeqn{EL1}, \refeqn{EL2}, and whatever thrust magnitude controller, e.g., equation \refeqn{af}, is selected. %These equations can be analyzed to determine the full translational motion of the quadrotor UAV during the attitude controlled flight mode.   

\section{POSITION CONTROLLED FLIGHT MODE}\label{sec:PCFM}

We now introduce a nonlinear controller for the position controlled flight mode. %This flight mode requires analysis of the coupled translational and rotational equations of motion; hence, we make use of the notation and analysis in the prior section to describe the properties of the closed loop system in this flight mode.  

\subsection{Position Tracking Errors}

Suppose that an arbitrary smooth position tracking command $x_d(t) \in \Re^3$ is given.    The position tracking errors for the position and the velocity are given by:
\begin{align}
e_x  = x - x_d,\quad
e_v  = \dot e_x = v - \dot x_d.
\end{align}
Similar with \refeqn{eI}, an integral control term for the position tracking controller is defined as
\begin{align}
e_i = \int_0^t e_v(\tau) + c_1 e_x (\tau) d\tau,
\end{align}
for a positive constant $c_1$ specified later. 

\newcommand{\sat}{\mathrm{sat}}
For a positive constant $\sigma\in\Re$, a saturation function $\sat_\sigma:\Re\rightarrow [-\sigma,\sigma]$ is introduced as
\begin{align*}
\sat_\sigma(y) = \begin{cases}
\sigma & \mbox{if } y >\sigma\\
y & \mbox{if } -\sigma \leq y \leq\sigma\\
-\sigma & \mbox{if } y <-\sigma\\
\end{cases}.
\end{align*}
If the input is a vector $y\in\Re^n$, then the above saturation function is applied element by element to define a saturation function $\sat_\sigma(y):\Re^n\rightarrow [-\sigma,\sigma]^n$ for a vector.

In the position controlled tracking mode, the attitude dynamics is controlled to follow the computed attitude $R_c(t)\in\SO$ and the computed angular velocity $\Omega_c(t)$ defined as
\begin{align}
R_c=[ b_{1_c};\, b_{3_c}\times b_{1_c};\, b_{3_c}],\quad \hat\Omega_c = R_c^T \dot R_c\label{eqn:RdWc},
\end{align}
where $b_{3_c}\in\Sph^2$ is given by
\begin{align}
 b_{3_c} = -\frac{-k_x e_x - k_v e_v -k_i\sat_\sigma(e_i) - mg e_3 +m\ddot x_d}{\norm{-k_x e_x - k_v e_v -k_i\sat_\sigma(e_i)- mg e_3 + m\ddot x_d}},\label{eqn:Rd3}
\end{align}
for positive constants $k_x,k_v,k_i,\sigma$. The unit vector $b_{1_c}\in\Sph^2$ is selected to be orthogonal to $b_{3_c}$, thereby guaranteeing that $R_c\in\SO$. It can be chosen to specify the desired heading direction, and the detailed procedure to select $b_{1c}$ is described later at Section \ref{sec:b1c}.

%Since $e_v=\dot e_x$, this can be rewritten as $e_i(t) = e_x(t) - e_x(0) + c_1 \int_0^t e_x(\tau)d\tau$. 

Following the prior definition of the attitude error and the angular velocity error, we define
\begin{align}
e_R = \frac{1}{2} (R_c^TR - R^T R_c)^\vee, \quad 
e_\Omega = \Omega - R^T R_c \Omega_c\label{eqn:eWc},
\end{align}
and we also define the integral term of the attitude dynamics $e_I$ as \refeqn{eI}.
It is assumed that
\begin{align}
\norm{-k_x e_x - k_v e_v -k_i\sat_\sigma(e_i) - mg e_3 + m\ddot x_d} \neq 0,\label{eqn:A1}
\end{align}
and the commanded acceleration is uniformly bounded:
\begin{align}
\|-mge_3+m\ddot x_d\| < B_1\label{eqn:B}
\end{align}
for a given positive constant $B_1$. It is also assumed that an upper bound of the infinite norm of the uncertainty is known:
\begin{align}
\|\Delta_x\|_{\infty} \leq \delta_x
\end{align}
for a given constant $\delta_x$. 

\subsection{Position Tracking Controller}

The nonlinear controller for the position controlled flight mode, described by control expressions for the  thrust magnitude and the  moment vector, are:
\begin{align}
f & = ( k_x e_x + k_v e_v +k_i\sat_\sigma(e_i)+ mg e_3-m\ddot x_d)\cdot Re_3,\label{eqn:f}\\
M & = -k_R e_R -k_\Omega e_\Omega -k_I e_I\nonumber\\
&\qquad +(R^TR_c\Omega_c)^\wedge J R^T R_c \Omega_c + J R^T R_c\dot\Omega_c.\label{eqn:M}
\end{align}

%The thrust magnitude controller and the moment vector controller is feedback dependent on the position and translational velocity and they depend on the commanded position, translational velocity and translational acceleration.    The control moment vector has a form that is similar to that for the attitude controlled flight mode.   However, the \textit{attitude error} and \textit{angular velocity error} are defined with respect to a computed attitude, angular velocity and angular acceleration, that are constructed according to the indicated procedure.  

The nonlinear controller given by equations \refeqn{f}, \refeqn{M} can be given a backstepping interpretation.   The computed attitude $R_c$ given in equation \refeqn{RdWc} is selected so that the thrust axis $-b_3$ of the quadrotor UAV tracks the computed direction given by $-b_{3_c}$ in \refeqn{Rd3}, which is a direction of the thrust vector that achieves position tracking.   The moment expression \refeqn{M} causes the attitude of the quadrotor UAV to asymptotically track $R_c$ and the thrust magnitude expression \refeqn{f} achieves asymptotic position tracking. The saturation on the integral term is required to restrict the effects of the attitude tracking errors on the translational dynamics for the stability of the complete coupled system.

The corresponding closed loop control system is described by equations \refeqn{EL1}-\refeqn{EL4}, using the controller expressions \refeqn{f}-\refeqn{M}. We now state the result that the zero equilibrium of tracking errors $(e_x, e_v, e_R, e_\Omega)$ is exponentially stable.

%\begin{figure}
%\centerline{
%\setlength{\unitlength}{2.1em}\centering\footnotesize
%\begin{picture}(11,3.2)(0.0,-3.2)
%\put(1.2,-1.5){\framebox(2,1.0)[c]%%
%{\shortstack[c]{Force\\controller}}}
%\put(4.0,-2.2){\framebox(2,1.0)[c]%%
%{\shortstack[c]{Moment\\controller}}}
%\put(0,-0.95){\vector(1,0){1.2}}
%\put(3.2,-1.35){\vector(1,0){0.8}}
%\put(3.2,-0.65){\vector(1,0){4.1}}
%\put(0,-1.8){\vector(1,0){4.0}}
%\put(6,-1.7){\vector(1,0){1.3}}
%\put(7.3,-2.25){\framebox(2.5,1.8)[c]%%
%{\shortstack[c]{Quadrotor\\Dynamics}}}
%\put(9.8,-1.35){\vector(1,0){1.2}}
%%
%\put(6.6,-1.0){$f$}
%\put(6.5,-2.05){$M$}
%\put(3.30,-1.15){$b_{3_c}$}
%\put(0.25,-0.80){$x_d$}
%\put(0.0,-1.65){($b_{1_d}$)}
%\put(5.5,-3.2){$x,v,R,\Omega$}
%%
%\put(10.3,-1.35){\line(0,-1){1.5}}
%\put(10.3,-2.85){\line(-1,0){8.1}}
%\put(2.2,-2.85){\vector(0,1){1.35}}
%\put(2.2,-2.10){\vector(1,0){1.8}}
%\put(2.2,-2.10){\circle*{0.1}}
%\put(10.3,-1.35){\circle*{0.1}}
%%
%\put(0.9,-2.4){\dashbox{0.08}(5.4,2.1)[c]{}}
%\put(2.8,-2.7){Controller}
%\end{picture}
%}
%\caption{Controller structure for position controlled flight mode}\label{fig:CS}
%\end{figure}

\begin{prop}{(Position Controlled Flight Mode)}\label{prop:Pos}
Suppose that the initial conditions satisfy
\begin{gather}
\Psi(R(0),R_c(0)) < \psi_1 < 1,\label{eqn:Psi0}\\
\|e_x(0)\| < e_{x_{\max}},\label{eqn:ex0}
\end{gather}
for positive constants $\psi_1, e_{x_{\max}}$. Consider the control inputs $f,M$ defined in \refeqn{f}-\refeqn{M}.
For positive constants $k_x,k_v$, we choose positive constants $c_1,c_2,k_R,k_\Omega,k_I,k_i,\sigma$ such that
\begin{gather}
k_i\sigma > \delta_x,\label{eqn:kisigma}\\
c_1 < \min\braces{\frac{4k_xk_v(1-\alpha)^2}{k_v^2(1+\alpha)^2+4m k_x(1-\alpha)},\; \sqrt{\frac{k_x}{m}} },\label{eqn:c1b}\\
\lambda_{m}(W_2) > \frac{\|W_{12}\|^2}{4\lambda_{m}(W_1)},\label{eqn:kRkWb}
\end{gather}
and \refeqn{c2} is satisfied, where $\alpha=\sqrt{\psi_1(2-\psi_1)}$, and the matrices $W_1,W_{12},W_2\in\Re^{2\times 2}$ are given by
\begin{align}
W_1 &= \begin{bmatrix} {c_1k_x}(1-\alpha) & -\frac{c_1k_v}{2}(1+\alpha)\\
-\frac{c_1k_v}{2}(1+\alpha) & k_v(1-\alpha)-mc_1\end{bmatrix},\label{eqn:W1}\\
W_{12}&=\begin{bmatrix}
{c_1}(\sqrt{3}k_i\sigma+B_1) & 0 \\ k_i\sigma+B_1+k_xe_{x_{\max}} & 0\end{bmatrix},\label{eqn:W12}\\
W_2 & = \begin{bmatrix} c_2k_R & -\frac{c_2}{2}(k_\Omega+B_2) \\ 
-\frac{c_2}{2}(k_\Omega+B_2) & k_\Omega-2c_2\lambda_M \end{bmatrix}.\label{eqn:W2}
\end{align}
Then, the zero equilibrium of the tracking errors is exponentially stable with respect to $e_x,e_v,e_R,e_\Omega$, and the integral terms $e_i,e_I$ are uniformly bounded.
\end{prop}

\begin{proof}
See Appendix \ref{sec:pfPos}.
%See \cite{LeeLeo_4457}.
\end{proof}

%This proposition shows that the proposed control system is robust to bounded, and unstructured uncertainties in the dynamics of a quadrotor UAV. Similar to Proposition \ref{prop:Att}, the ultimate bound can be arbitrarily reduced by choosing smaller $\epsilon_x,\epsilon_R$, and it is possible to obtain exponential attractiveness.

Proposition \ref{prop:Pos} requires that the initial attitude error is less than $90^\circ$ in \refeqn{Psi0}. Suppose that this is not satisfied, i.e. $1\leq\Psi(R(0),R_c(0))<2$. We can still apply Proposition \ref{prop:Att}, which states that the attitude error is asymptotically decreases to zero for almost all cases, and it satisfies \refeqn{Psi0} in a finite time. Therefore, by combining the results of Proposition \ref{prop:Att} and \ref{prop:Pos}, we can show attractiveness of the tracking errors when $\Psi(R(0),R_c(0))<2$.

\begin{prop}{(Position Controlled Flight Mode with a Larger Initial Attitude Error)}\label{prop:Pos2}
Suppose that the initial conditions satisfy
\begin{gather}
1\leq \Psi(R(0),R_c(0)) < 2\label{eqn:eRb3},\\
\|e_x(0)\| < e_{x_{\max}},
\end{gather}
for a constant $e_{x_{\max}}$. Consider the control inputs $f,M$ defined in \refeqn{f}-\refeqn{M}, where the control parameters satisfy \refeqn{kisigma}-\refeqn{kRkWb} for a positive constant $\psi_1<1$. Then the zero equilibrium of the tracking errors is attractive, i.e., $e_x,e_v,e_R,e_\Omega\rightarrow 0$ as $t\rightarrow\infty$. 
\end{prop}

\begin{proof}
See Appendix \ref{sec:pfPos2}.
%See \cite{LeeLeo_4457}.
\end{proof}

Linear or nonlinear PID controllers have been widely used for a quadrotor UAV. But, they have been applied in an \textit{ad-hoc} manner without stability analysis. This paper provides a new form of nonlinear PID controller on $\SE$ that guarantees almost global attractiveness in the presence of uncertainties. Compared with nonlinear robust control system~\cite{LeeLeoPACC12}, this paper yields stronger asymptotic stability without concern for chattering.

%in various experimental settings for a quadrotor UAV, without careful stability analyses. This paper provides a new form of integral control terms that guarantees asymptotic convergence of tracking errors with uncertainties. 

%The nonlinear robust tracking control system in~\cite{LeeLeoPACC12} provides ultimate boundedness of tracking errors, and the control input may be prone to chattering if the required ultimate bound is smaller. Compared with~\cite{LeeLeoPACC12}, the control system in this paper stronger asymptotic stability, and there is no concern for discontinuities. The structure of the control system is also simplified such that the cross term of the angular velocity does not have to be cancelled. 

\subsection{Direction of the First Body-Fixed Axis}\label{sec:b1c}

As described above, the construction of the orthogonal matrix $R_c$ involves having its third column $b_{3_c}$ specified by \refeqn{Rd3}, and its first column $b_{1_c}$ is arbitrarily chosen to be orthogonal to the third column, which corresponds to a one-dimensional degree of choice. 

%. The unit vector $b_{1_c}$ can be arbitrarily chosen in the plane normal to $b_{3_c}$, which corresponds to a one-dimensional degree of choice. 

%This reflects the fact that the quadrotor UAV has four control inputs that are used to track a three-dimensional position command.

By choosing $b_{1_c}$ properly, we constrain the asymptotic direction of the first body-fixed axis. 
Here, we propose to specify the \textit{projection} of the first body-fixed axis onto the plane normal to $b_{3_c}$. In particular, we choose a desired direction $b_{1_d}\in\Sph^2$, that is not parallel to $b_{3_c}$, and $b_{1_c}$ is selected as $b_{1_c}=\mathrm{Proj}[b_{1_d}]$, where $\mathrm{Proj}[\cdot]$ denotes the normalized projection onto the plane perpendicular to $b_{3_c}$. In this case, the first body-fixed axis does not converge to $b_{1_d}$, but it converges to the projection of $b_{1_d}$, i.e. $b_1\rightarrow b_{1_c}=\mathrm{Proj}[b_{1_d}]$ as $t\rightarrow\infty$. This can be used to specify the heading direction of a quadrotor UAV in the horizontal plane~\cite{LeeLeo}.
\section{NUMERICAL EXAMPLE}\label{sec:NE}
%%%
%Numerical results are presented to demonstrate the prior approach for performing complex flight maneuvers. 

The parameters of a quadrotor UAV are chosen as $J=[0.43,0.43,1.02]\times 10^{-2}\,\mathrm{kgm^2}$, $m=0.755\,\mathrm{kg}$, $d=0.169\,\mathrm{m}$, $c_{\tau f}=0.0132\,\mathrm{m}$. Disturbances for the translational dynamics and the rotational dynamics are chosen as
\begin{align*}
\Delta_x= [-0.5,
    0.2,
    1]^T\,\mathrm{N},\quad
\Delta_R = [0.2,-0.1,-0.02]^T\,\mathrm{Nm}.
\end{align*}
Controller parameters are selected as follows: $k_x=12.8$, $k_v=4.22$, $k_i=1.28$, $k_R=0.65$, $k_\Omega = 0.11$, $k_I=0.06$, $c_1=3.6$, $c_2=0.8$, $\sigma=1$.

Initially, the quadrotor UAV is at a hovering condition: $x(0)=v(0)=\Omega(0)=0_{3\times 1}$, and $R(0)=I_{3\times 3}$. The desired trajectory is a flipping maneuver where the quadrotor rotates about its second body-fixed axis by $360^\circ$, while changing the heading angle by $90^\circ$ about the vertical $e_3$ axis. This is a complex maneuver combining a nontrivial pitching maneuver with a yawing motion. It is achieved by concatenating the following two control modes:
\begin{itemize}
\item[(i)] Attitude tracking to rotate the quadrotor ($t\leq 0.4$)
\begin{align*}
R_d(t)= \exp(\pi t\hat e_3)\exp(4\pi t\hat e_2).
\end{align*}
\item[(ii)] Trajectory tracking to make it hover after completing the preceding rotation ($0.4<t \leq 4$)
\begin{align*}
x_d(t)=[0,0,0]^T,\quad b_{1_d} = [0,1,0]^T.
\end{align*}
\end{itemize}

Figure \ref{fig:NoInt} illustrate simulation results without the integral control terms proposed in this paper. There are steady state errors in attitude tracking and position tracking at Figures \ref{fig:NoIntPsi} and \ref{fig:NoIntx}. The proposed integral control terms eliminate the steady state error while exhibiting good tracking performances as shown at Figure \ref{fig:Int}. The resulting controlled maneuver of the  quadrotor UAV is illustrated at Figure \ref{fig:Int3D}. 

In the prior results of generating nontrivial maneuvers of a quadrotor UAV,  complicated reachability analyses are required to guarantee safe transitions between multiple control systems~\cite{GilHofIJRR11}. In the proposed geometric nonlinear control system, there are only two controlled flight modes for position tracking and attitude tracking, and each controller has large region of attraction. Therefore, complex maneuvers can be easily generated in a unified way without need for time-consuming planning efforts, as illustrated by this numerical example. This is another unique contribution of this paper.

\begin{figure}
\centerline{
	\subfigure[Attitude error function $\Psi$]{
		\includegraphics[width=0.48\columnwidth]{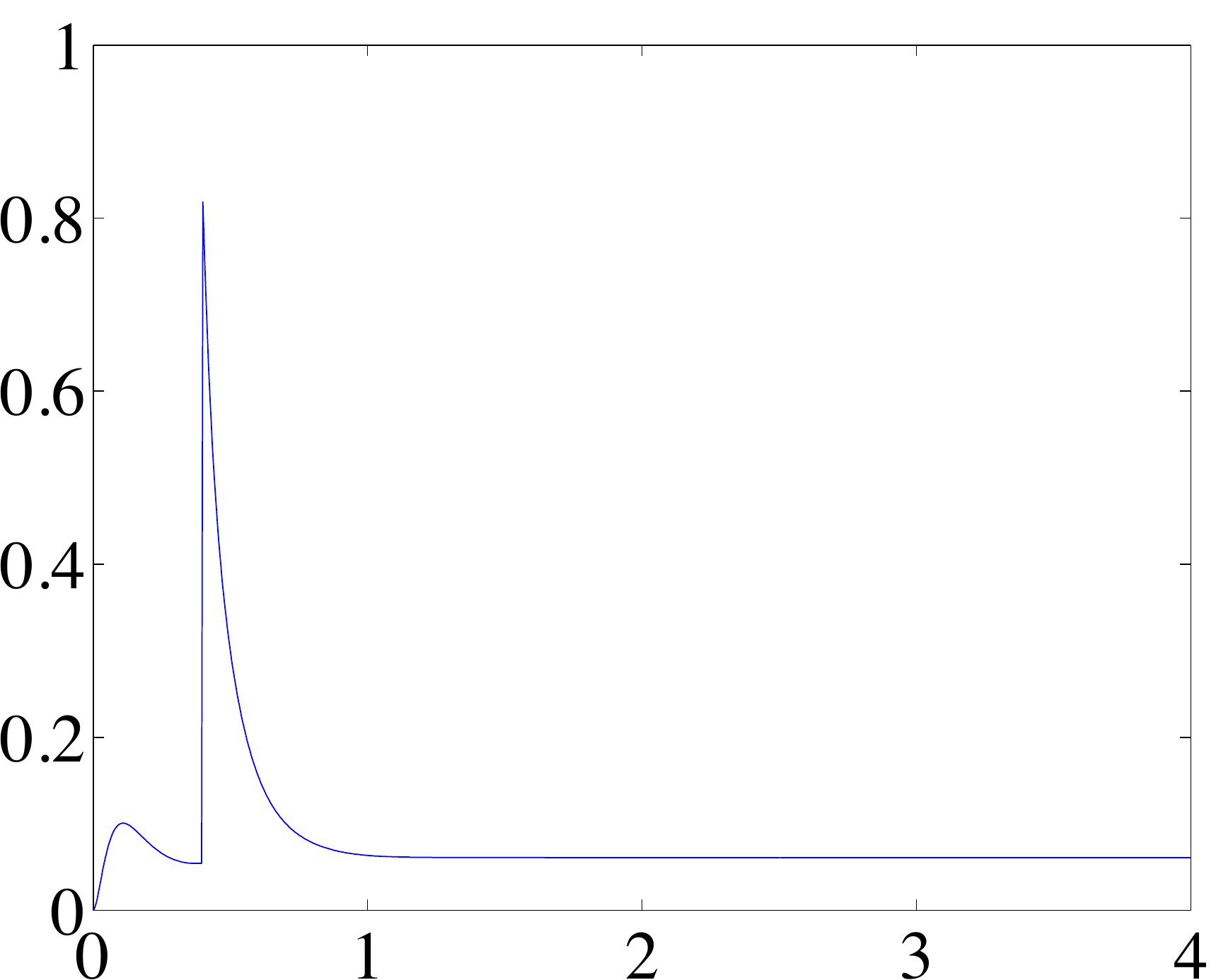}\label{fig:NoIntPsi}}
		\hfill
	\subfigure[Position $x,x_d$ ($\mathrm{m}$)]{
		\includegraphics[width=0.485\columnwidth]{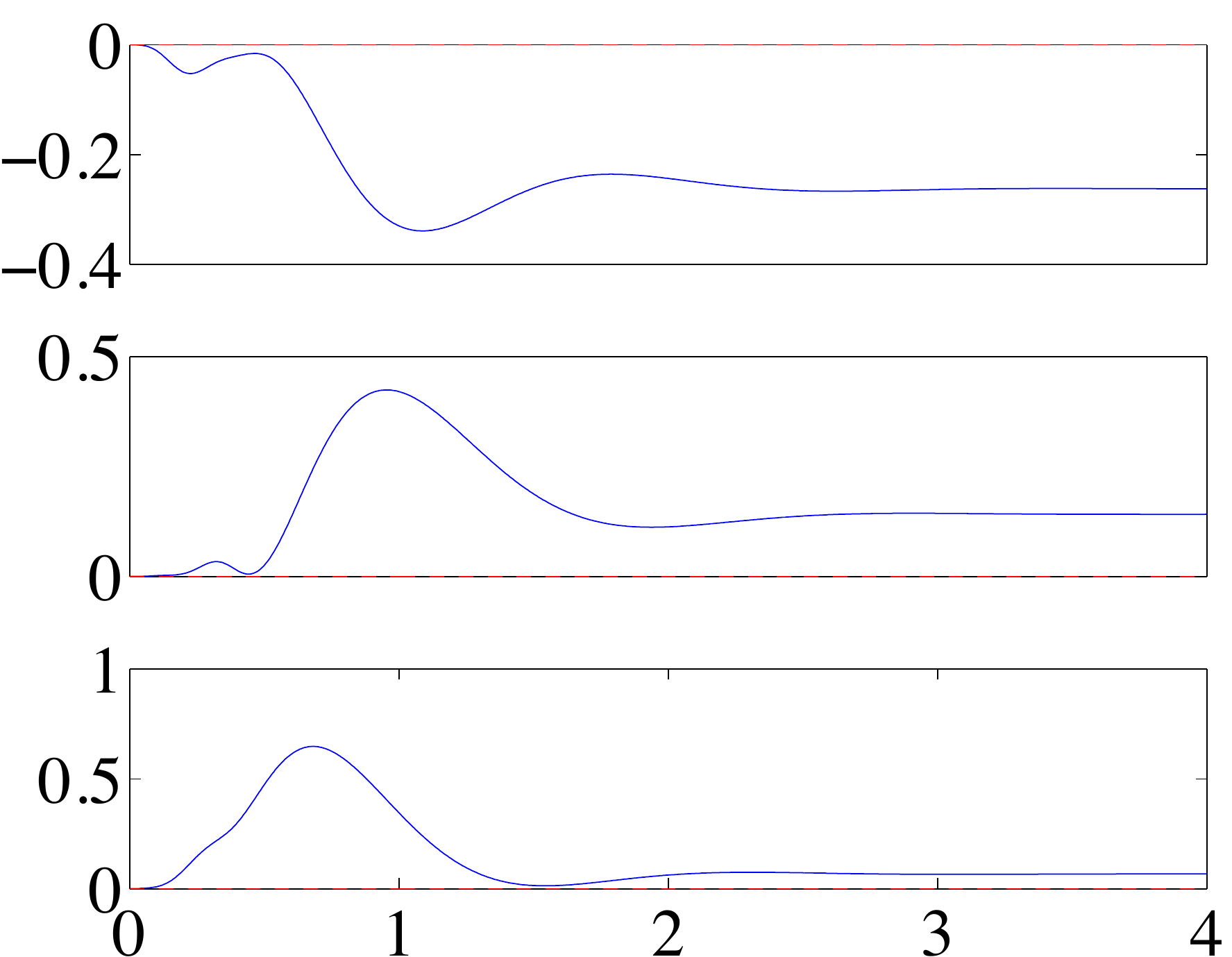}\label{fig:NoIntx}}
}
\centerline{
	\subfigure[Angular velocity $\Omega,\Omega_d$ ($\mathrm{/sec}$)]{
		\includegraphics[width=0.48\columnwidth]{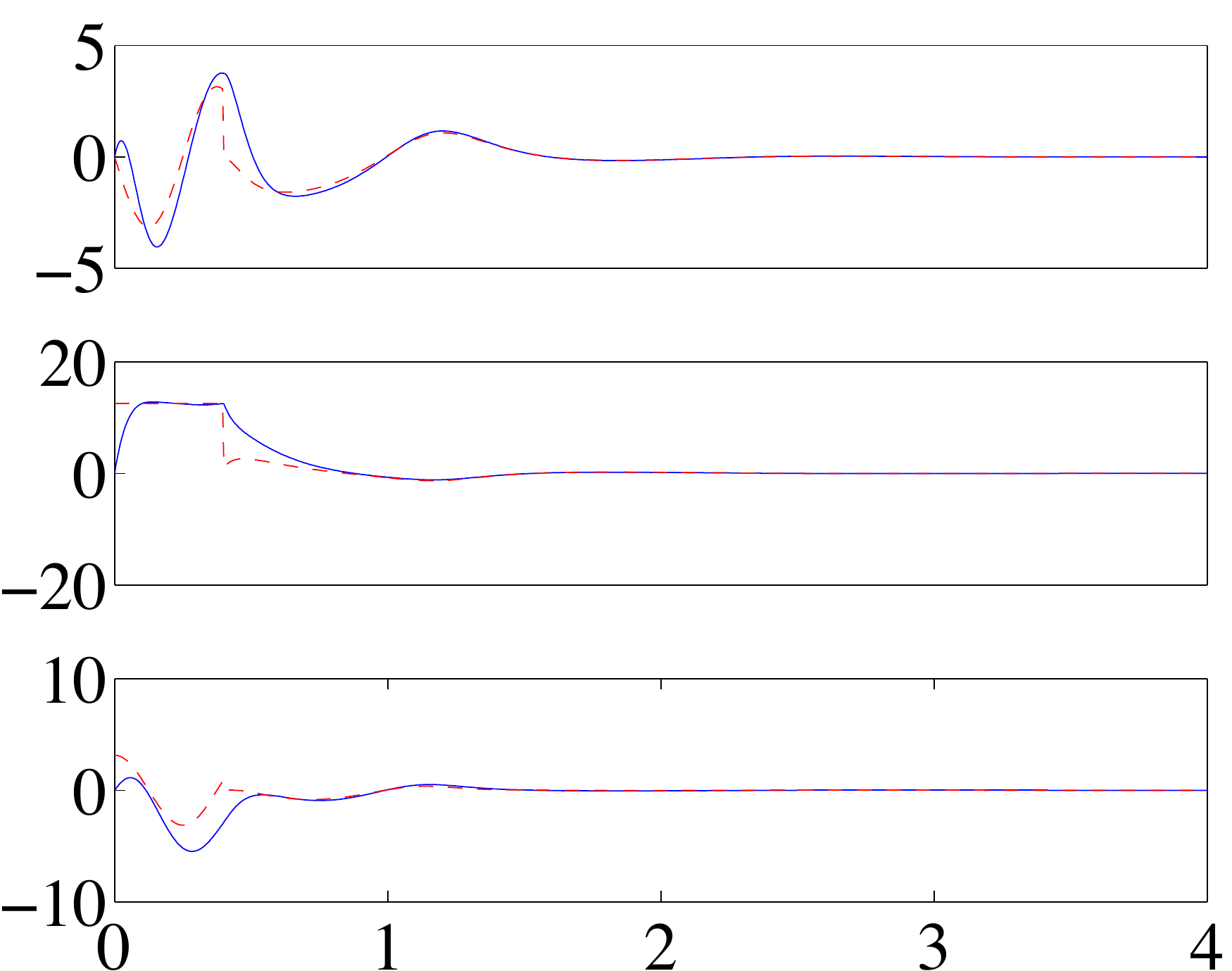}\label{fig:NoIntW}}
		\hfill
	\subfigure[Thrust of each rotor ($\mathrm{N}$)]{
		\includegraphics[width=0.50\columnwidth]{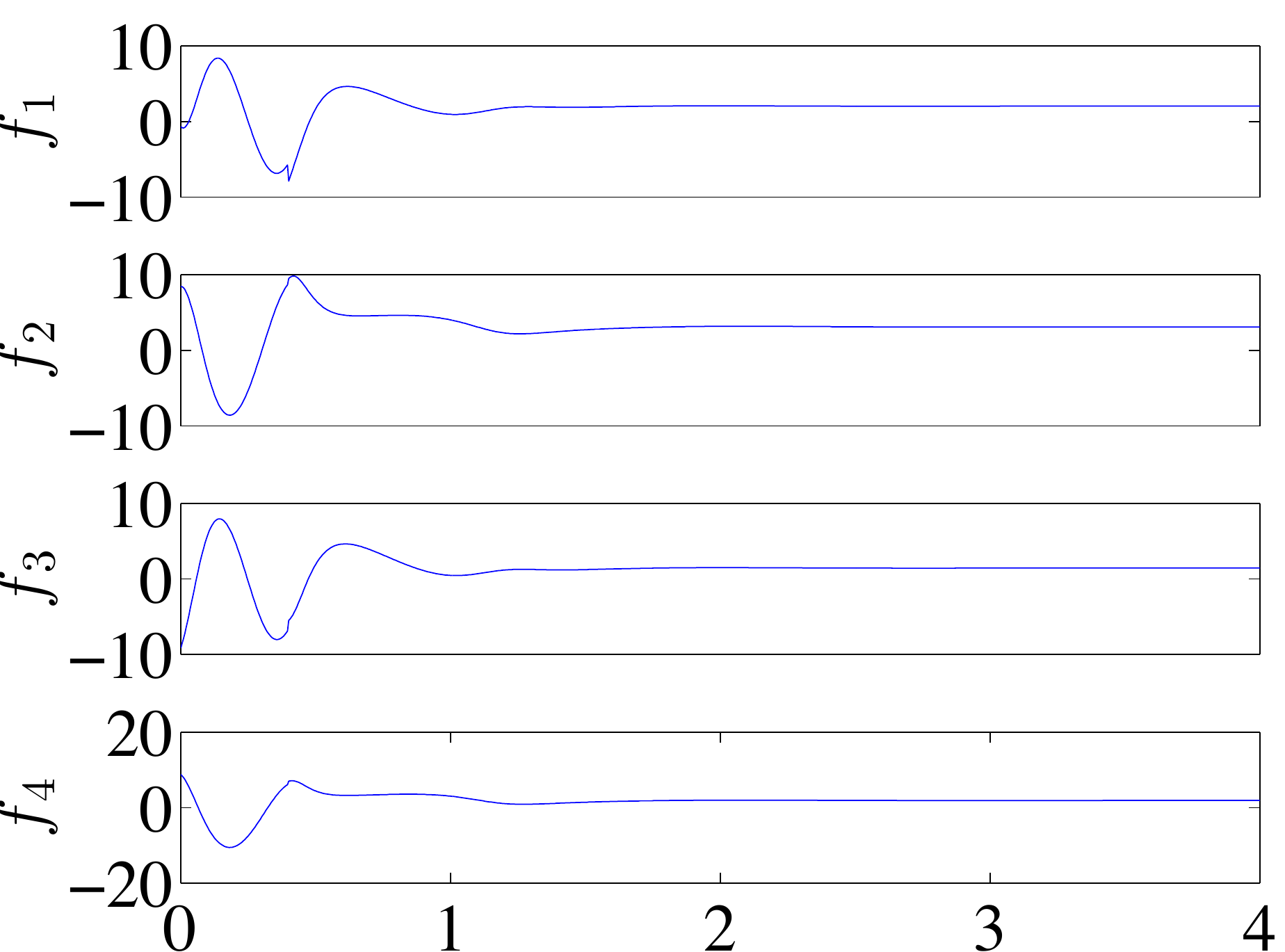}}
}
\caption{Flipping without integral terms (red,dotted:desired, blue,solid:actual)}\label{fig:NoInt}
\end{figure}

\begin{figure}
\centerline{
	\subfigure[Attitude error function $\Psi$]{
		\includegraphics[width=0.48\columnwidth]{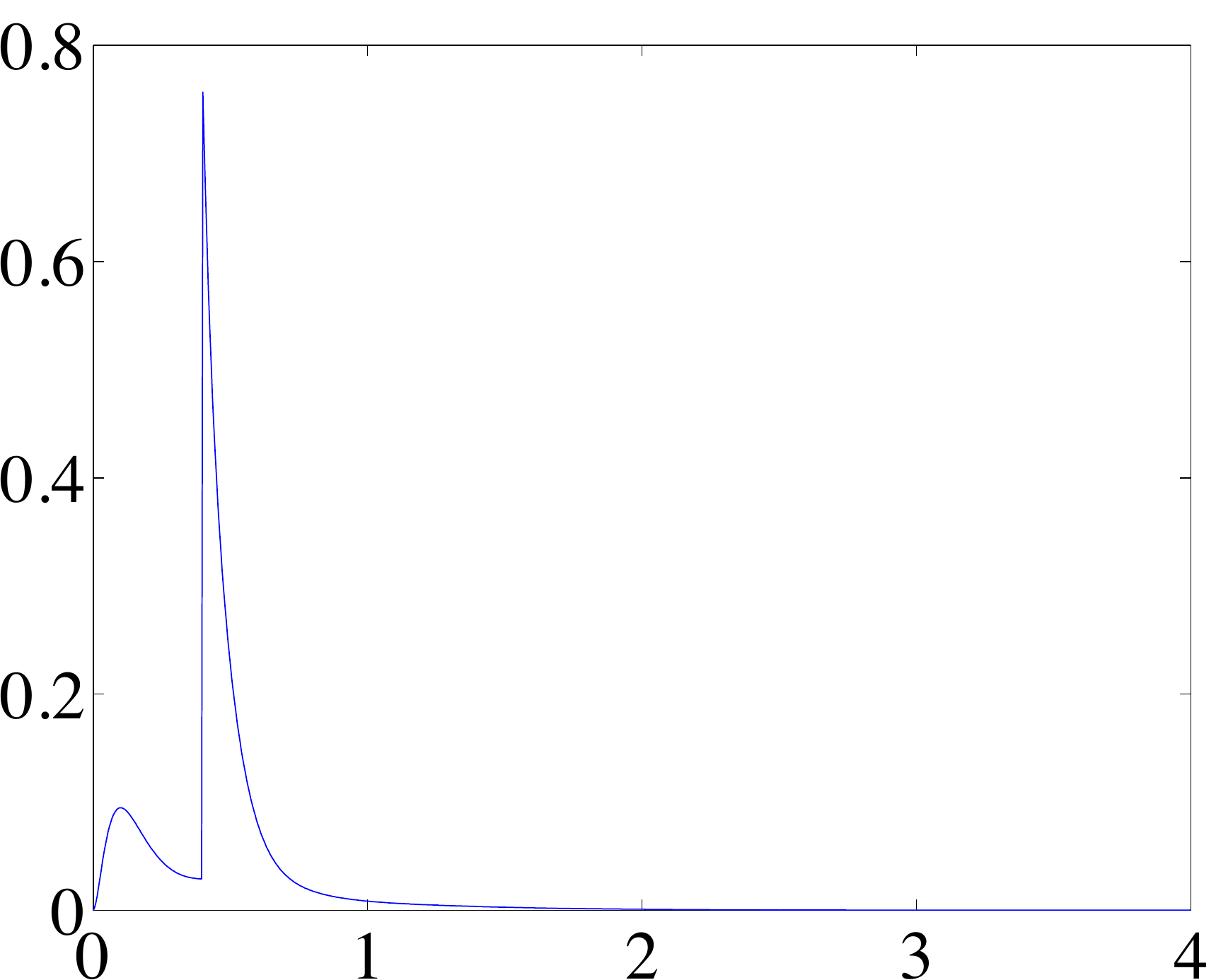}\label{fig:IntPsi}}
		\hfill
	\subfigure[Position $x,x_d$ ($\mathrm{m}$)]{
		\includegraphics[width=0.485\columnwidth]{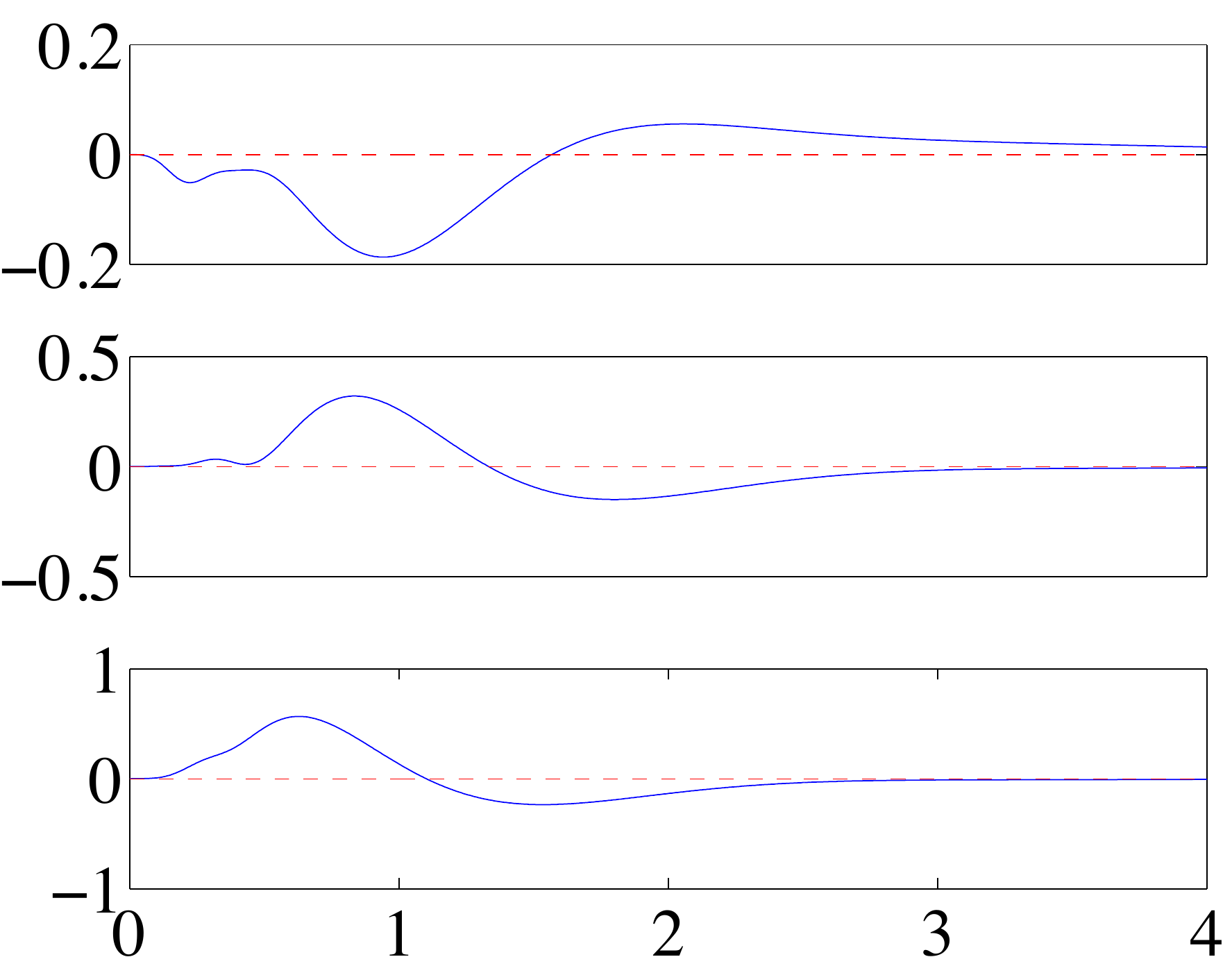}\label{fig:Intx}}
}
\centerline{
	\subfigure[Angular velocity $\Omega,\Omega_d$ ($\mathrm{/sec}$)]{
		\includegraphics[width=0.48\columnwidth]{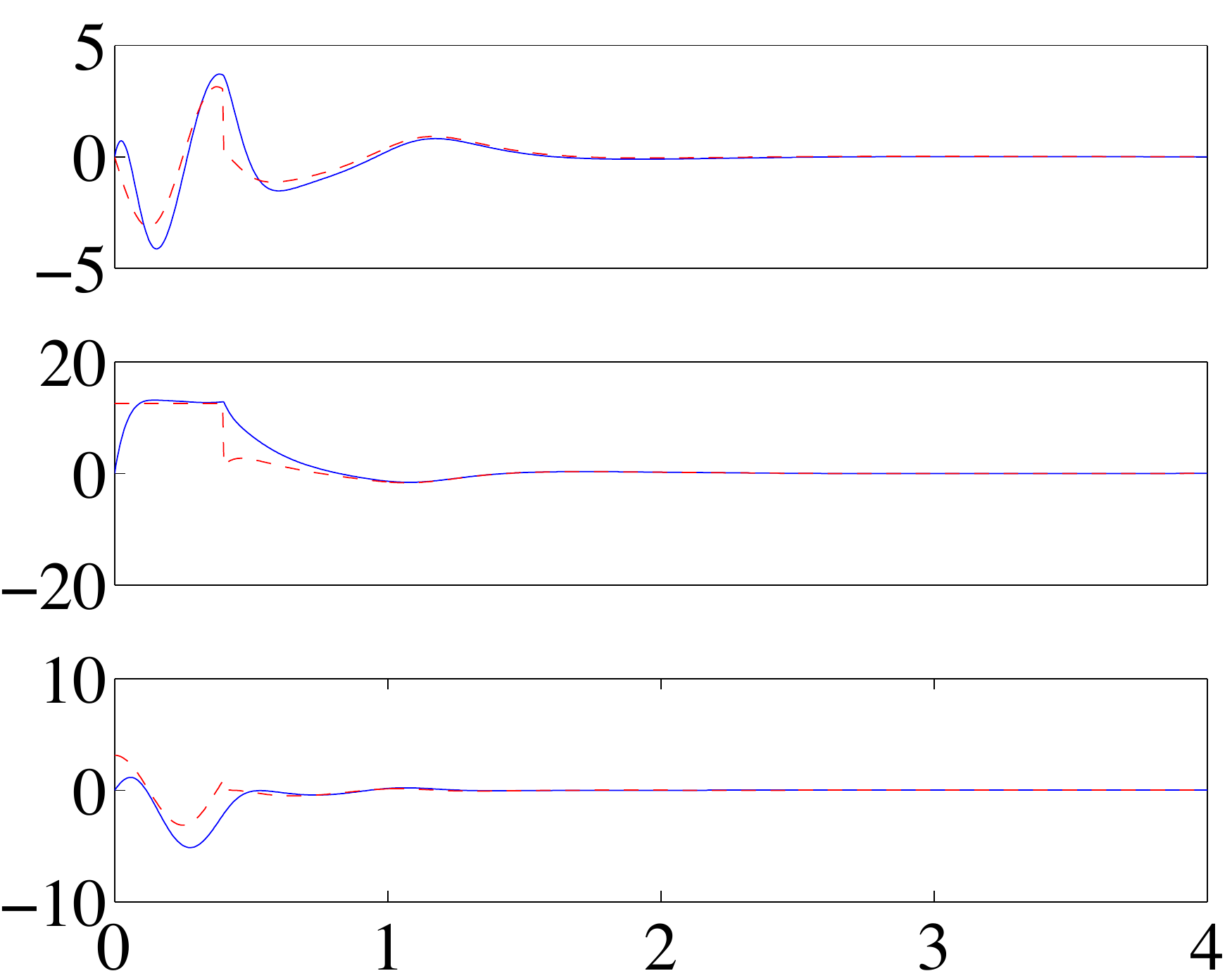}\label{fig:IntW}}
		\hfill
	\subfigure[Thrust of each rotor ($\mathrm{N}$)]{
		\includegraphics[width=0.50\columnwidth]{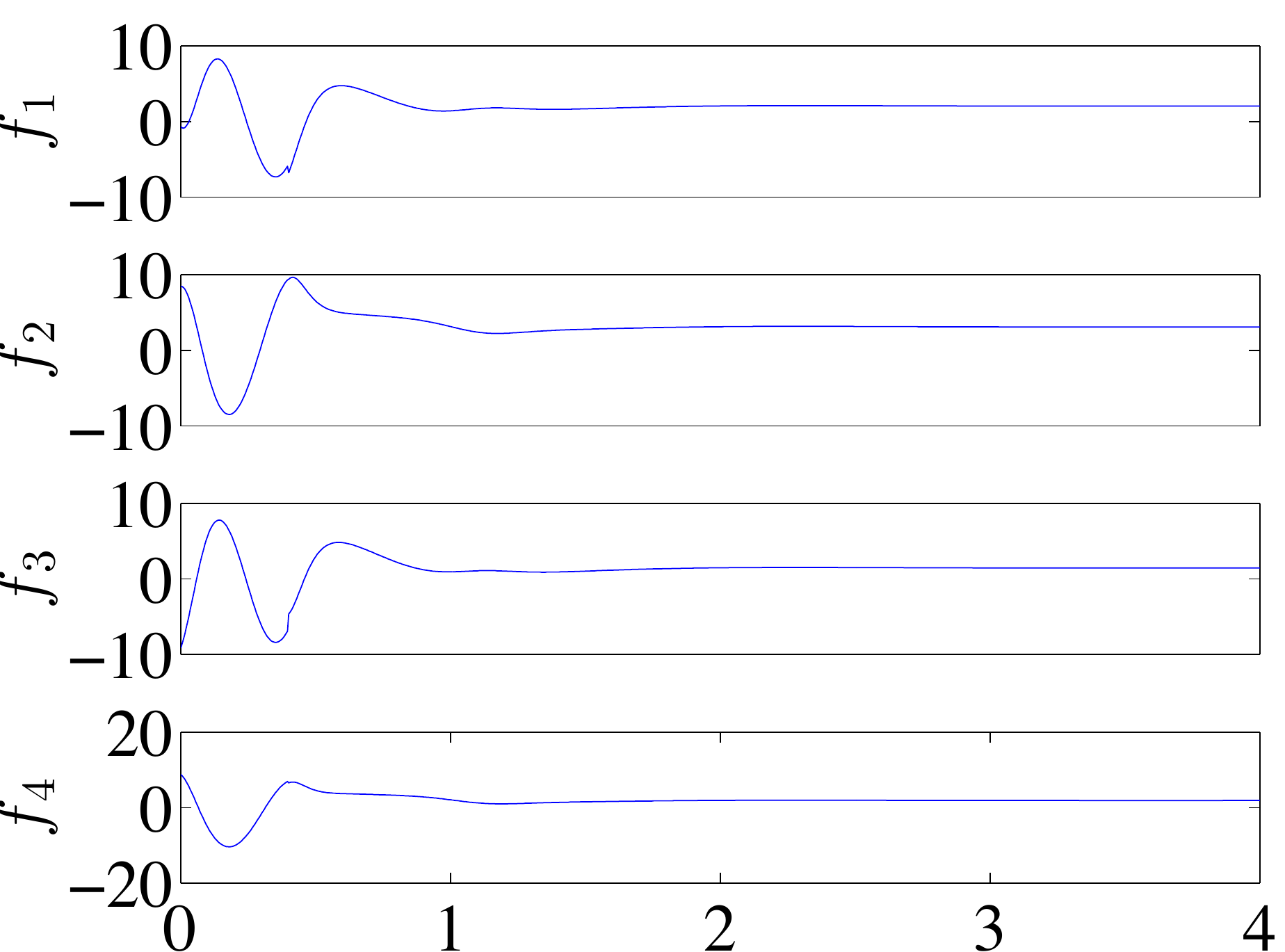}}
}
\caption{Flipping with integral terms (red,dotted:desired, blue,solid:actual)}
\label{fig:Int}
\end{figure}

%\begin{figure}
%\centerline{
%%\setlength{\unitlength}{0.1\columnwidth}\scriptsize
%%\begin{picture}(7,4)(0,0)
%\setlength{\unitlength}{0.1\columnwidth}\scriptsize
%\begin{picture}(6.3,4.5)(0,0)
%	\put(0,0){\includegraphics[width=0.63\columnwidth]{Int3D.pdf}}
%	\put(1.01,4.2){$\vec e_1$}
%	\put(0.98,3.4){$\vec e_2$}
%	\put(0.25,3.35){$\vec e_3$}
%\end{picture}}
%\caption{Snapshots of a flipping maneuver with integral terms: the red axis denotes the direction of the first body-fixed axis. The quadrotor UAV rotates about the horizontal $e_2$ axis by $360^\circ$, while rotating its first body-fixed axis about the vertical $e_3$ axis by $90^\circ$. The trajectory of its mass center is denoted by blue, dotted lines.}\label{fig:Int3D}
%\end{figure}

\begin{figure}
\centerline{
%\setlength{\unitlength}{0.1\columnwidth}\scriptsize
%\begin{picture}(7,4)(0,0)
\setlength{\unitlength}{0.1\columnwidth}\scriptsize
\begin{picture}(9.9,4.5)(0,0)
	\put(0,0){\includegraphics[width=0.99\columnwidth]{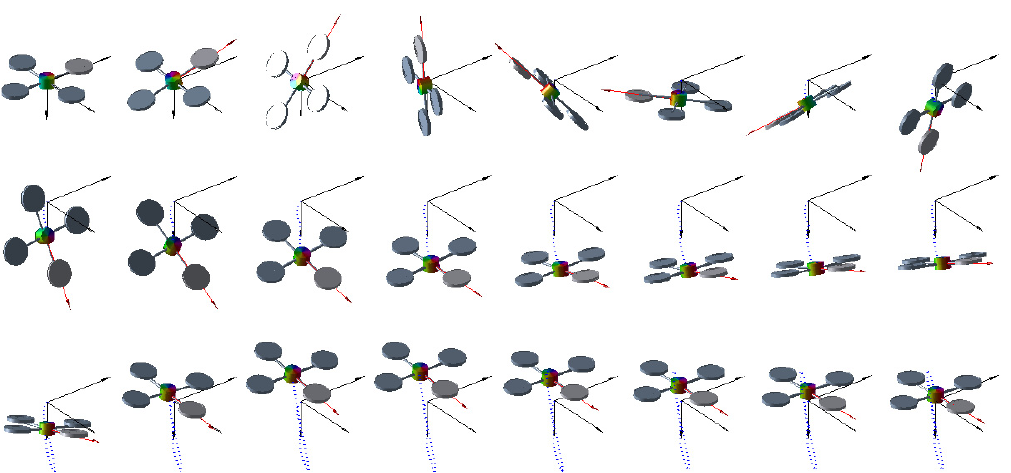}}
	\put(0.98,4.2){$\vec e_1$}
	\put(0.98,3.3){$\vec e_2$}
	\put(0.25,3.25){$\vec e_3$}
\end{picture}}
\caption{Snapshots of a flipping maneuver with integral terms: the red axis denotes the direction of the first body-fixed axis. The quadrotor UAV rotates about the horizontal $e_2$ axis by $360^\circ$, while rotating its first body-fixed axis about the vertical $e_3$ axis by $90^\circ$. The trajectory of its mass center is denoted by blue, dotted lines.}\label{fig:Int3D}
\end{figure}
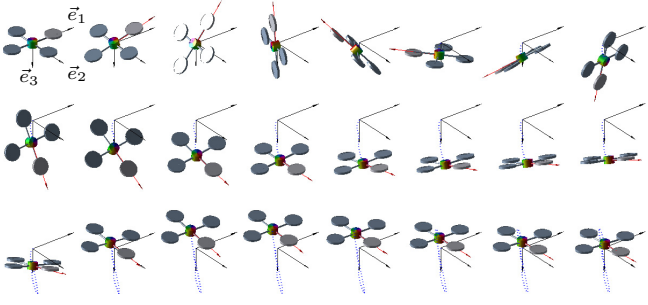

\section{Preliminary Experimental Results}

Preliminary experimental results are provided for the attitude tracking control of a hardware system illustrated at Figure \ref{fig:Quad}. To test the attitude dynamics, it is attached to a spherical joint. As the center of rotation is below the center of gravity, there exists a destabilizing gravitational moment, and the resulting attitude dynamics is similar to an inverted rigid body pendulum. The control input at \refeqn{aM} is augmented with an additional term to eliminate the effects of the gravity.

The desired attitude command is described by using 3-2-1 Euler angles, i.e. $R_d(t)=R_d(\phi(t),\theta(t),\psi(t))$, where
$\phi(t) =  \frac{\pi}{9}\sin(\pi t)$, $\theta(t)=\frac{\pi}{9}\cos(\pi t)$, $\psi(t)=0$. This represents a combined rolling and pitching motion with a period of $2$ seconds. The results of the experiment are illustrated at Figure \ref{fig:QuadResult}. This shows good tracking performances of the proposed control system in an experimental setting. Experiments for the position tracking is currenlty ongoing.

\begin{figure}
\centerline{
	\subfigure[Hardware configuration]{
\setlength{\unitlength}{0.1\columnwidth}\scriptsize
\begin{picture}(7,4)(0,0)
\put(0,0){\includegraphics[width=0.7\columnwidth]{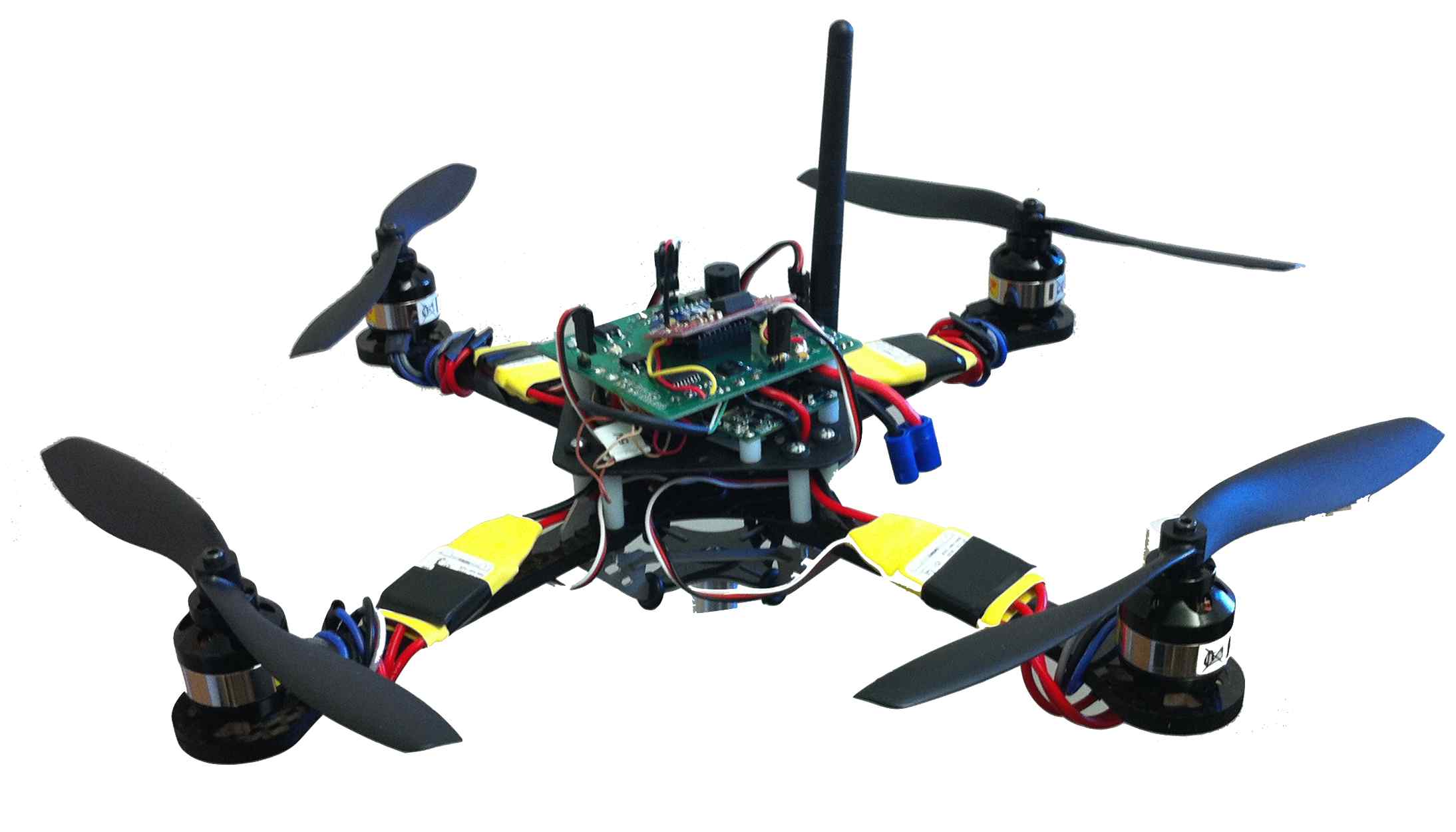}}
\put(1.95,3.2){\shortstack[c]{OMAP 600MHz\\Processor}}
\put(2.3,0){\shortstack[c]{Attitude sensor\\3DM-GX3\\ via UART}}
\put(0.85,1.4){\shortstack[c]{BLDC Motor\\ via I2C}}
\put(0.1,2.5){\shortstack[c]{Safety Switch\\XBee RF}}
\put(4.3,3.2){\shortstack[c]{WIFI to\\Ground Station}}
\put(5,2.0){\shortstack[c]{LiPo Battery\\11.1V, 2200mAh}}
\end{picture}}
	\subfigure[Attitude control testbed]{
	\includegraphics[width=0.27\columnwidth]{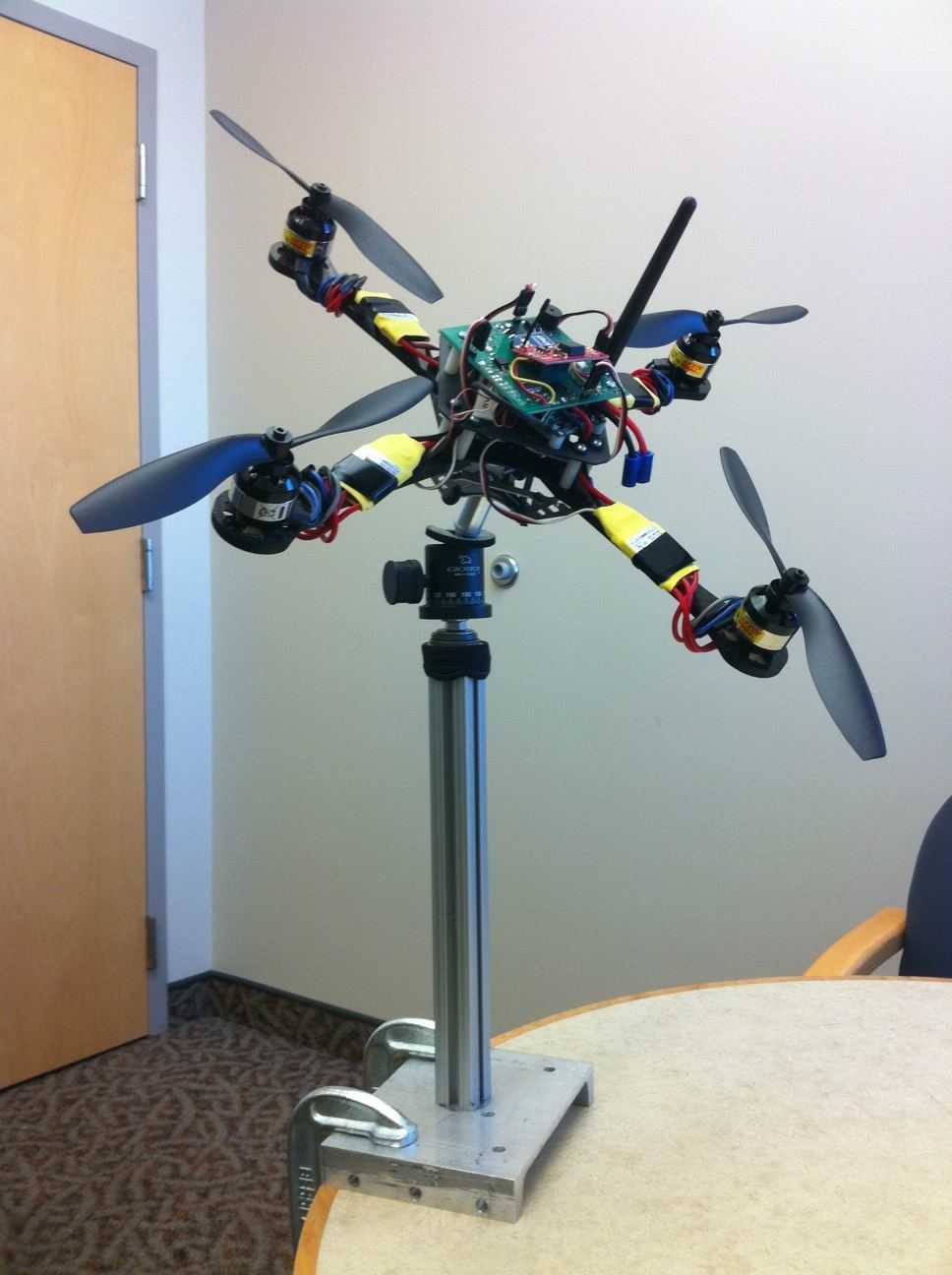}}
}
\caption{Hardware development for a quadrotor UAV}\label{fig:Quad}
\end{figure}
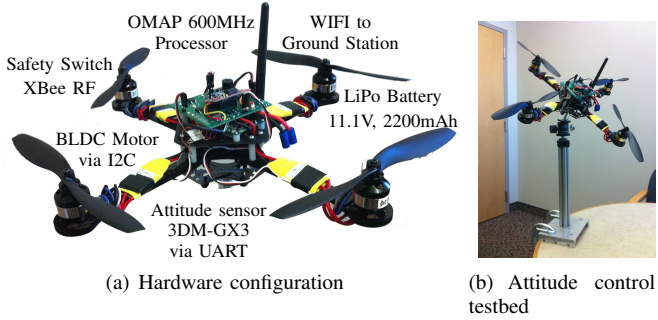

\begin{figure}
\centerline{
	\subfigure[Attitude error function $\Psi$]{
		\includegraphics[width=0.48\columnwidth]{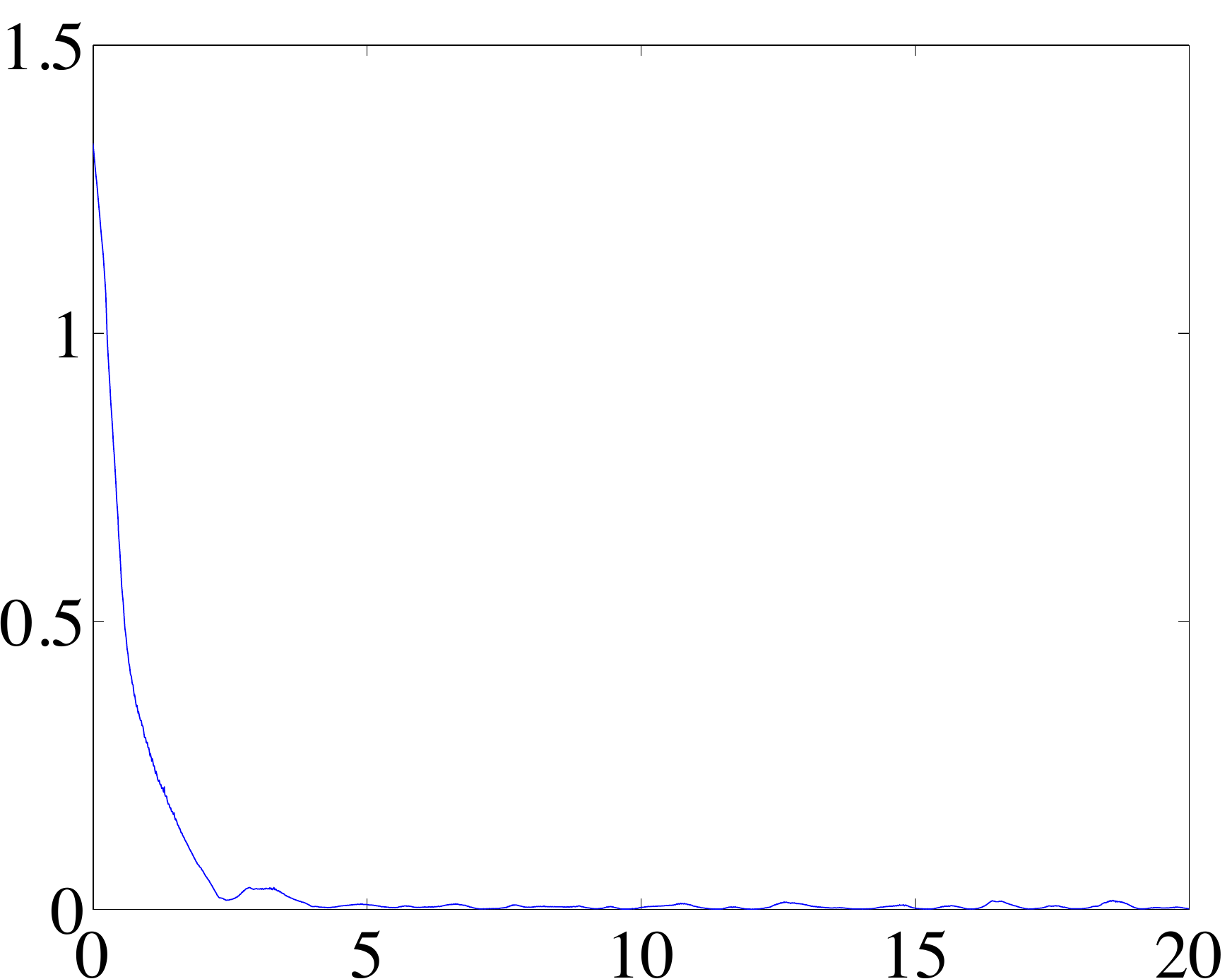}\label{fig:QuadPsi}}
		\hspace*{0.01\textwidth}
	\subfigure[Attitude $R,R_d$]{
		\includegraphics[width=0.47\columnwidth]{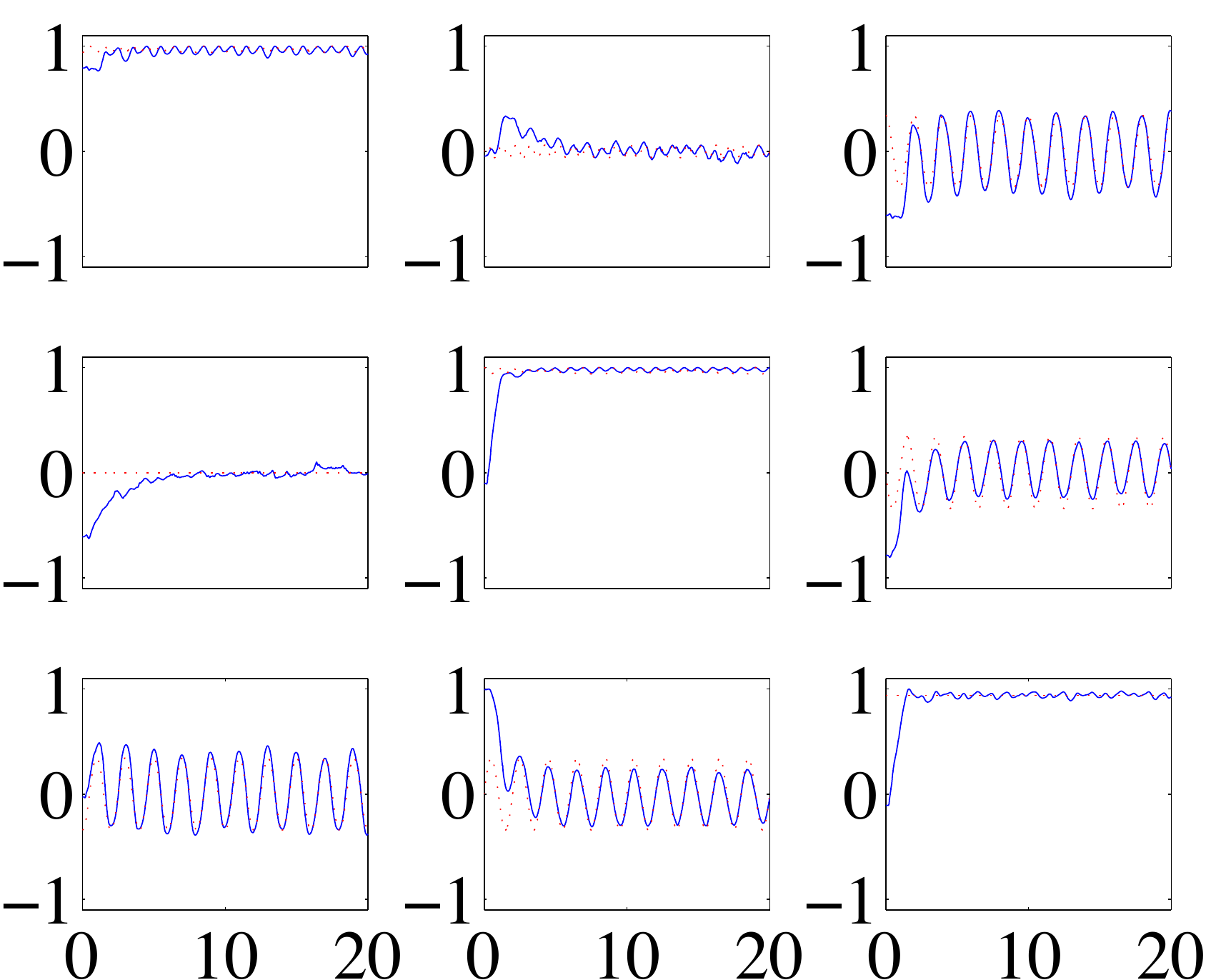}\label{fig:Quadx}}
}
\centerline{
	\subfigure[Angular velocity $\Omega,\Omega_d$ ($\mathrm{/sec}$)]{
		\includegraphics[width=0.48\columnwidth]{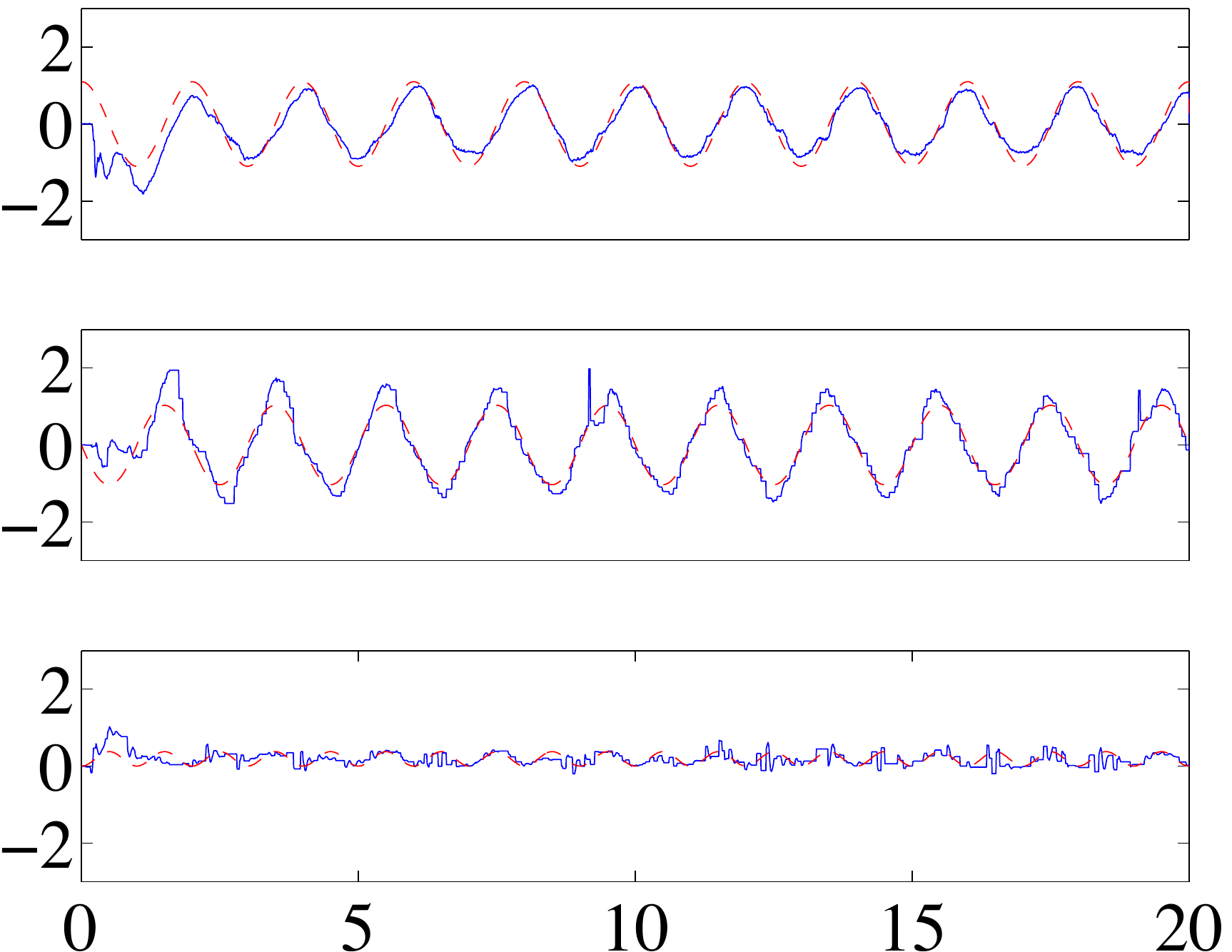}\label{fig:QuadW}}
		\hfill
	\subfigure[Thrust of each rotor ($\mathrm{N}$)]{
		\includegraphics[width=0.50\columnwidth]{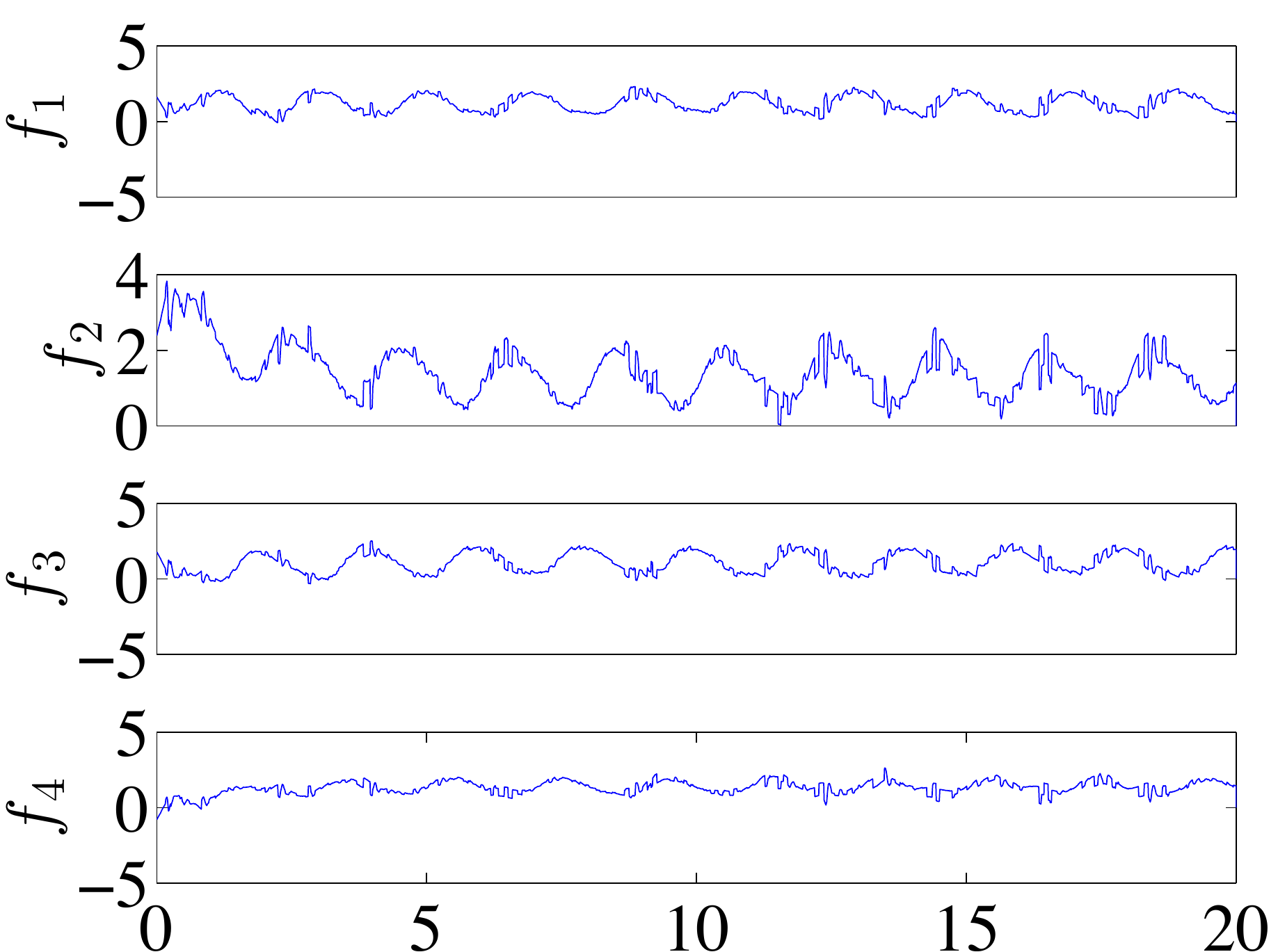}}
}
\caption{Attitude tracking experiment (red,dotted:desired, blue,solid:actual)}\label{fig:QuadResult}
\end{figure}

%We apply the robust adaptive attitude control system at Proposition \ref{prop:RAC} to this quadrotor UAV. The control input at \refeqn{u2} is augmented with an additional term to eliminate the gravitational moment. The disturbances are mainly due to the error in canceling the gravitational moment, the friction in the spherical joint, as well as sensor noises and thrust measurement errors. 

\appendix

\section{Properties and Proofs}

%\subsection{Properties of the \textit{Hat} Map}\label{app:hat}
%\noindent The hat map $\hat\cdot :\Re^3\rightarrow\so$ is defined as
%\begin{align}
%    \hat x = \begin{bmatrix} 0 & -x_3 & x_2\\
%                                x_3 & 0 & -x_1\\
%                                -x_2 & x_1 & 0 \end{bmatrix}\label{eqn:hat}
%\end{align}
%for $x=[x_1;x_2;x_3]\in\Re^3$. This identifies the Lie algebra $\so$ with $\Re^3$ using the vector cross product in $\Re^3$. The inverse of the hat map is referred to as the \textit{vee} map, $\vee:\so\rightarrow\Re^3$. 

%Several properties of the hat map are summarized as follows.
%\begin{gather}
%    \hat x y = x\times y = - y\times x = - \hat y x,\\
%%    \hat x^T \hat x = (x^T x) I - x x^T,\\
%%    \hat x \hat y \hat x=-(y^Tx)\hat x,\\
%    -\frac{1}{2}\tr{\hat x \hat y} = x^T y,\\
%%    \widehat{x\times y} = \hat x \hat y -\hat y \hat x = yx^T-xy^T,\\
%    \tr{\hat x A}=\tr{A\hat x }=\frac{1}{2}\tr{\hat x (A-A^T)}=-x^T (A-A^T)^\vee,\label{eqn:hat1}\\
%%    \widehat{Ax} = \hat x \parenth{\frac{1}{2}\tr{A}I-A} + \parenth{\frac{1}{2}\tr{A}I-A}^T \hat x,\\
%    \hat x  A+A^T\hat x=(\braces{\tr{A}I_{3\times 3}-A}x)^{\wedge},\label{eqn:xAAx}\\
%R\hat x R^T = (Rx)^\wedge,\label{eqn:hat2}
%\end{gather}
%for any $x,y\in\Re^3$, $A\in\Re^{3\times 3}$, and $R\in\SO$.

\subsection{Proof of Proposition \ref{prop:Att}}\label{sec:pfAtt}

We first find the error dynamics for $e_R,e_\Omega$, and define a Lyapunov function. Then, we find conditions on control parameters to guarantee the boundedness of tracking errors.

Using \refeqn{EL3}, \refeqn{EL4}, \refeqn{M}, the time-derivative of $Je_\Omega$ can be written as
\begin{align}
J\dot e_\Omega & = \{Je_\Omega + d\}^\wedge e_\Omega - k_R e_R-k_\Omega e_\Omega- k_I e_I + \Delta_R,\label{eqn:JeWdot}
\end{align}
where $d=(2J-\trs{J}I)R^TR_d\Omega_d\in\Re^3$. The important property is that the first term of the right hand side is normal to $e_\Omega$, and it simplifies the subsequent Lyapunov analysis.

Define a Lyapunov function $\mathcal{V}_2$ be 
\begin{align}
\mathcal{V}_2 & = \frac{1}{2} e_\Omega \cdot J e_\Omega + k_R\, \Psi(R,R_d)+c_2 e_R\cdot Je_\Omega\nonumber\\
&\quad +\frac{1}{2}k_I \|e_I-\frac{\Delta_R}{k_I}\|^2.\label{eqn:V2}
\end{align}

From \refeqn{PsiLB}, \refeqn{PsiUB}, the Lyapunov function $\mathcal{V}_2$ is bounded as
\begin{align}
z_2^T M_{21} z_2
&+\frac{k_I}{2} \|e_I-\frac{\Delta_R}{k_I}\|^2 
\leq \mathcal{V}_2 \nonumber\\
&\leq z_2^T M_{22} z_2+\frac{k_I}{2} \|e_I-\frac{\Delta_R}{k_I}\|^2,
\label{eqn:V2b}
\end{align}
where $z_2 =[\|e_R\|,\;\|e_\Omega\|]^T\in\Re^2$, and the matrices $M_{12},M_{22}$ are given by
\begin{align}
M_{21} = \frac{1}{2}\begin{bmatrix} k_R & -c_2\lambda_M \\ -c_2\lambda_M & \lambda_m  \end{bmatrix},\,
M_{22} = \frac{1}{2}\begin{bmatrix} \frac{2k_R}{2-\psi_2} & c_2\lambda_M \\ c_2\lambda_M & \lambda_{M}\end{bmatrix}.
%\label{eqn:M2}
\end{align}
From \refeqn{PsiUB}, the upper-bound of \refeqn{V2b} is satisfied in the following domain:
\begin{align}
D_2 = \{ (R,\Omega)\in \SO\times\Re^3\,|\, \Psi(R,R_d)<\psi_2<2\}.\label{eqn:D2}
\end{align}

From \refeqn{Psidot00}, \refeqn{JeWdot}, the time derivative of $\mathcal{V}_2$ along the solution of the controlled system is given by
\begin{align*}
\dot{\mathcal{V}}_2 & =
-k_\Omega\|e_\Omega\|^2 - e_\Omega\cdot(k_Ie_I-\Delta_R) + c_2 \dot e_R \cdot Je_\Omega\nonumber\\
& \quad + c_2 e_R \cdot J\dot e_\Omega + (k_Ie_I- \Delta_R) \dot e_I.
\end{align*}
From \refeqn{eI}, we have $\dot e_I = c_2 e_R + e_\Omega$. Substituting this and \refeqn{JeWdot}, the above equation becomes
\begin{align*}
\dot{\mathcal{V}}_2 & =
-k_\Omega\|e_\Omega\|^2  + c_2 \dot e_R \cdot Je_\Omega\nonumber\\
& \quad -c _2 k_R \|e_R\|^2 + c_2 e_R \cdot ((Je_\Omega+d)^\wedge e_\Omega -k_\Omega e_\Omega).
\end{align*}
Since $\|e_R\|\leq 1$, $\|\dot e_R\|\leq \|e_\Omega\|$, and $\|d\|\leq B_2$, we have
\begin{align}
\dot{\mathcal{V}}_2 \leq - z_2^T W_2 z_2,\label{eqn:dotV2}
\end{align}
where the matrix $W_2\in\Re^{2\times 2}$ is given by
\begin{align*}
W_2 = \begin{bmatrix} c_2k_R & -\frac{c_2}{2}(k_\Omega+B_2) \\ 
-\frac{c_2}{2}(k_\Omega+B_2) & k_\Omega-2c_2\lambda_M \end{bmatrix}.%\label{eqn:W2}
\end{align*}
The condition on $c_2$ given at \refeqn{c2} guarantees that all of matrices $M_{21},W_2$ are positive definite. This implies that the zero equilibrium of tracking errors $(e_R,e_\Omega,e_I)=(0,0,\frac{\Delta_R}{k_I})$ is stable in the sense of Lyapunov, and $e_R,e_\Omega\rightarrow 0$ as $t\rightarrow\infty$. But, this does not necessarily implies $R\rightarrow R_d$ as $e_R=0$ at any critical point of $\Psi$ described at the property (iii) of Proposition \ref{prop:1}. 

Instead, we show instability of undesired equilibrium. Define $\mathcal{W}_2=2k_R-\mathcal{V}_2$. Then $\mathcal{W}_2=0$ at the undesired equilibria as $\Phi=2$ at those points. Since
\begin{align*}
\mathcal{W}_2 &\geq -\frac{\lambda_M}{2}\|e_\Omega\|^2 +k_R(2-\Psi) -c_2\|e_R\|\|e_\Omega\| \\
&\quad -\frac{k_I}{2}\|e_I-\frac{\Delta_R}{k_I}\|^2.
\end{align*}
Due to the continuity of $\Psi$, at any arbitrary small neighborhood of the undesired equilibrium attitude, we can choose $R$ such that $2-\Psi>0$. Therefore, if $\|e_\Omega\|$ and $\|e_I-\frac{\delta_R}{k_I}\|$ are sufficiently small, then $\mathcal{W}_2>0$ at such attitudes. In short, at any arbitrary small neighborhood of the undesired equilibrium, there exists a domain where $\mathcal{W}_2>0$, and $\dot{\mathcal{W}}_2=-\dot{\mathcal{V}}_2>0$ in that domain from \refeqn{dotV2}. According to Theorem 4.3 at~\cite{Kha02}, the undesired equilibrium is unstable.

The region of attraction to the desired equilibrium excludes the stable manifolds to the undesired equilibria. But the dimension of the union of the stable manifolds to the unstable equilibria is less than the tangent bundle of $\SO$. Therefore, the measure of the stable manifolds to the unstable equilibrium is zero. Then, the desired equilibrium is referred to as almost globally asymptotically stable with respect to $e_R$ and $e_\Omega$.

If this Lyapunov analysis is restricted to the domain $D_2$, then the upper-bound of \refeqn{V2b} is satisfied, and the condition on $c_2$ guarantees that the matrix $M_{22}$ is positive definite. These yield local exponential stability with respect to $e_R$ and $e_\Omega$.

\subsection{Proof of Proposition \ref{prop:Pos}}\label{sec:pfPos}
\setcounter{paragraph}{0}

We first derive the tracking error dynamics and a Lyapunov function for the translational dynamics of a quadrotor UAV, and later it is combined with the stability analyses of the rotational dynamics in Appendix \ref{sec:pfAtt}.

The subsequent analyses are developed in the domain $D_1$
\begin{align}
D_1=\{&(e_x,e_v,R,e_\Omega)\in\Re^3\times\Re^3\times \SO\times\Re^3\,|\,\nonumber\\
& \|e_x\|< e_{x_{\max}},\;\Psi< \psi_1 < 1\},\label{eqn:D}
\end{align}
Similar to \refeqn{PsiUB}, we can show that 
\begin{align}
\frac{1}{2} \norm{e_R}^2 \leq  \Psi(R,R_c) \leq \frac{1}{2-\psi_1} \norm{e_R}^2\label{eqn:eRPsi1}.
\end{align}

\paragraph{Translational Error Dynamics} The time derivative of the position error is $\dot e_x=e_v$. The time-derivative of the velocity error is given by
\begin{align}
m\dot e_v = m\ddot x -m\ddot x_d = mg e_3 - fRe_3 -m\ddot x_d+\Delta_x. \label{eqn:evdot0}
\end{align}
Consider the quantity $e_3^T R_c^T R e_3$, which represents the cosine of the angle between $b_3=Re_3$ and $b_{3_c}=R_ce_3$. Since $1-\Psi(R,R_c)$ represents the cosine of the eigen-axis rotation angle between $R_c$ and $R$, we have $e_3^T R_c^T R e_3\geq  1-\Psi(R,R_c)>0$ in $D_1$. Therefore, the quantity $\frac{1}{e_3^T R_c^T R e_3}$ is well-defined. To rewrite the error dynamics of $e_v$ in terms of the attitude error $e_R$, we add and subtract $\frac{f}{e_3^T R_c^T R e_3}R_c e_3$ to the right hand side of \refeqn{evdot0} to obtain
\begin{align}
m\dot e_v &  = mg e_3 -m\ddot x_d- \frac{f}{e_3^T R_c^T R e_3}R_c e_3 - X+\Delta_x,\label{eqn:evdot1}
%&\quad - \frac{f}{e_3^T R_d^T R e_3}( (e_3^T R_d^T R e_3)R e_3 -R_de_3).
\end{align}
where $X\in\Re^3$ is defined by
\begin{align}
X=\frac{f}{e_3^T R_c^T R e_3}( (e_3^T R_c^T R e_3)R e_3 -R_ce_3).\label{eqn:X}
\end{align}
Let $A=-k_x e_x - k_v e_v -k_i\sat_\sigma(e_i)-mg e_3 + m\ddot x_d$.
%be the desired control force for the translational dynamics. 
Then, from \refeqn{Rd3}, \refeqn{f}, we have ${b}_{3_c}=R_c e_3 = -A/\norm{A}$ and $f=-A\cdot Re_3$. By combining these, we obtain $f= (\norm{A}R_c e_3)\cdot R e_3$. Therefore, the third term of the right hand side of \refeqn{evdot1} can be written as
\begin{align*}
-  \frac{f}{e_3^T R_c^T R e_3} & R_c e_3 = -\frac{(\norm{A}R_c e_3)\cdot R e_3}{e_3^T R_c^T R e_3}\cdot - \frac{A}{\norm{A}}=A\\
& =-k_x e_x - k_v e_v -k_i\sat_\sigma(e_i) -mg e_3 + m\ddot x_d.
\end{align*}
Substituting this into \refeqn{evdot1}, the error dynamics of $e_v$ can be written as
\begin{align}
m\dot e_v & =  -k_x e_x - k_v e_v -k_i\sat_\sigma(e_i) - X+\Delta_x.\label{eqn:evdot}
\end{align}

\paragraph{Lyapunov Candidate for Translation Dynamics}
Let a Lyapunov candidate $\mathcal{V}_1$ be
\begin{align}
\mathcal{V}_1 & = \frac{1}{2}k_x\|e_x\|^2  + \frac{1}{2} m \|e_v\|^2 + c_1 e_x\cdot me_v\nonumber\\
&\quad
+ \int_{\frac{\Delta_x}{k_i}}^{e_i} (k_i\sat_\sigma (\mu)-\Delta_x)\cdot d\mu
\label{eqn:V1}.
\end{align}
The condition given at \refeqn{kisigma} implies that the last integral term of the above equation is positive definite about $e_i=\frac{\Delta_x}{k_i}$. The derivative of ${\mathcal{V}}_1$ along the solution of \refeqn{evdot} is given by
\begin{align}
\dot{\mathcal{V}}_1 
& =  -(k_v-mc_1) \|e_v\|^2 
- c_1 k_x \|e_x\|^2 
-c_1 k_v e_x\cdot e_v\nonumber\\
&\quad+X\cdot \braces{ c_1 e_x + e_v}.\label{eqn:V1dot0}
\end{align}

The last term of the above equation corresponds to the effects of the attitude tracking error on the translational dynamics. We find a bound of $X$, defined at \refeqn{X}, to show stability of the coupled translational dynamics and rotational dynamics in the subsequent Lyapunov analysis. Since $f=\|A\| (e_3^T R_c^T R e_3)$, we have
\begin{align*}
\norm{X} & \leq \|A\|\,\| (e_3^T R_c^T R e_3)R e_3 -R_ce_3\|\\
& \leq( k_x \|e_x\| + k_v \|e_v\| + \sqrt{3}k_i\sigma +B_1)\\
&\quad \times \| (e_3^T R_c^T R e_3)R e_3 -R_ce_3\|.
\end{align*}
The last term $\| (e_3^T R_c^T R e_3)R e_3 -R_ce_3\|$ represents the sine of the angle between $b_3=Re_3$ and $b_{c_3}=R_c e_3$, since $(b_{3_c}\cdot b_3)b_3 - b_{3_c} = b_{3}\times (b_3\times b_{3_c})$. The magnitude of the attitude error vector, $\|e_R\|$ represents the sine of the eigen-axis rotation angle between $R_c$ and $R$ (see \cite{LeeLeo}). Therefore, $\| (e_3^T R_c^T R e_3)R e_3 -R_ce_3\| \leq \| e_R\|$ in $D_1$. It follows that 
\begin{align}
\| (e_3^T R_d^T R e_3)R e_3 & -R_de_3\| \leq \| e_R\| = \sqrt{\Psi(2-\Psi)}\nonumber\\
& \leq \braces{\sqrt{\psi_1 (2-\psi_1)}\triangleq\alpha}  <1.\label{eqn:eR_bound}
\end{align}
Therefore, $X$ is bounded by
\begin{align}
\norm{X} 
&\leq ( k_x \|e_x\| + k_v \|e_v\| + \sqrt{3}k_i\sigma + B_1) \|e_R\| \nonumber\\
&\leq ( k_x \|e_x\| + k_v \|e_v\| + \sqrt{3}k_i\sigma + B_1) \alpha.\label{eqn:XB}
\end{align}
Substituting \refeqn{XB} into \refeqn{V1dot0}, 
\begin{align}
\dot{\mathcal{V}}_1 %& \leq   -(k_v-c_1) \|e_v\|^2 
%- \frac{c_1 k_x}{m} \|e_x\|^2 
%-\frac{c_1k_v}{m}e_x\cdot e_v\nonumber\\
%& \quad +( k_x \|e_x\| + k_v \|e_v\| + B+\delta_x) \|e_R\| \braces{ \frac{c_1}{m} \|e_x\| + \|e_v\|}+\epsilon_x\nonumber\\
& \leq   -(k_v(1-\alpha)-mc_1) \|e_v\|^2 
- {c_1 k_x}(1-\alpha) \|e_x\|^2 \nonumber\\
& + {c_1k_v}(1+\alpha) \|e_x\|\|e_v\|\nonumber\\
& +  \|e_R\| \braces{(\sqrt{3}k_i\sigma +B_1)({c_1} \|e_x\| + \|e_v\|)+k_x\|e_x\|\|e_v\|}.\label{eqn:V1dot1}
\end{align}
In the above expression for $\dot{\mathcal{V}}_1$, there is a third-order error term, namely $k_x\|e_R\|\|e_x\|\|e_v\|$. Using \refeqn{eR_bound}, it is possible to choose its upper bound as $k_x\alpha\|e_x\|\|e_v\|$ similar to other terms, but the corresponding stability analysis becomes complicated, and the initial attitude error should be reduced further. Instead, we restrict our analysis to the domain $D_1$ defined in \refeqn{D}, and its upper bound is chosen as $k_xe_{x_{\max}}\|e_R\|\|e_v\|$.

%\paragraph{Boundedness of $\|e_v\|$} In the above expression for $\dot{\mathcal{V}}_1$, there is a third-order error term, namely $k_x\|e_R\|\|e_x\|\|e_v\|$. Here, we find a bound on $\|e_v\|$ to change this term into a second-order error term for the subsequent Lyapunov analysis. We consider a special case where the constants $c_1$ and $k_x$ are zero. Define $\mathcal{V}'_1=\mathcal{V}_1\big|_{c_1=0}$. From \refeqn{V1}, \refeqn{V1dot1}, we have
%\begin{align*}
%\mathcal{V}'_1 & = \frac{1}{2}k_x\|e_x\|^2 + \frac{1}{2} m \|e_v\|^2,\\
%\dot{\mathcal{V}}'_1 & 
%\leq   -k_v(1-\alpha) \|e_v\|^2 
%+  B\alpha\|e_v\| + k_x\alpha\|e_x\|\|e_v\|.
%\end{align*}
%Since $\|e_x\|\leq \sqrt{\frac{2}{k_x}\mathcal{V}'_1}$ and $\|e_v\|\leq \sqrt{\frac{2}{m}\mathcal{V}'_1}$, we have
%\begin{align*}
%\dot{\mathcal{V}}'_1 & 
%\leq   -k_v(1-\alpha) \|e_v\|^2
%+  B\alpha\sqrt{\frac{2}{m}\mathcal{V}'_1} + 2\alpha\sqrt{\frac{k_x}{m}} \mathcal{V}'_1.
%\end{align*}
%
%
%This implies that when $\|e_v\| > \frac{B}{k_v(1-\alpha)}$, the time derivative of $\|e_v\|$ is negative, and  $\|e_v\|$ monotonically decreases. Therefore, $\|e_v\|$ is uniformly bounded as
%\begin{align}
%\|e_v(t)\| \leq \max\braces{\|e_v(0)\|,\,\frac{B}{k_v(1-\alpha)}} \equiv e_{v_{\max}}.\label{eqn:evmax}
%\end{align}

\paragraph{Lyapunov Candidate for the Complete System}
Let $\mathcal{V}=\mathcal{V}_1+\mathcal{V}_2$ be the Lyapunov candidate of the complete system. Define $z_1=[\|e_x\|,\;\|e_v\|]^T$, $z_2=[\|e_R\|,\;\|e_\Omega\|]^T\in\Re^2$, and 
\begin{align*}
\mathcal{V}_I = k_i\int_{\frac{\Delta_x}{k_i}}^{e_i} (\sat_\sigma(\mu)-\Delta_x)\cdot d\mu + \frac{k_I}{2}\|e_I-\frac{\Delta_R}{k_I}\|^2.
\end{align*}
%\begin{align}
%\mathcal{V} & = \frac{1}{2} k_x \|e_x\|^2 + \frac{1}{2}m \|e_v\|^2 + c_1 e_x\cdot e_v\nonumber\\
%&\quad + \frac{1}{2}e_\Omega \cdot Je_  \Omega + k_R\Psi(R,R_d) + c_2 e_R\cdot e_\Omega.\label{eqn:V}
%\end{align}
Using \refeqn{eRPsi1}, the bound of the Lyapunov candidate $\mathcal{V}$ can be written as
\begin{align}
z_1^T M_{11} z_1 + z_2^T M_{21} z_2& + \mathcal{V}_I
 \leq \mathcal{V} \nonumber\\
&\leq z_1^T M_{12} z_1 + z_2^T M'_{22} z_2 +\mathcal{V}_I,\label{eqn:Vb}
\end{align}
where the matrices $M_{11},M_{12},M_{21},M_{22}$ are given by
\begin{gather*}
M_{11} = \frac{1}{2}\begin{bmatrix} k_x & -mc_1 \\ -mc_1 & m\end{bmatrix},\;
M_{12} = \frac{1}{2}\begin{bmatrix} k_x & mc_1 \\ mc_1 & m\end{bmatrix},\\
M_{21} = \frac{1}{2}\begin{bmatrix} k_R & -c_2\lambda_M \\ -c_2\lambda_M & \lambda_{m}  \end{bmatrix},\;
M'_{22} = \frac{1}{2}\begin{bmatrix} \frac{2k_R}{2-\psi_1} & c_2\lambda_M \\ c_2\lambda_M & \lambda_{M}\end{bmatrix}.
%\label{eqn:M2p}
\end{gather*}

Using \refeqn{dotV2} and \refeqn{V1dot1}, the time-derivative of $\mathcal{V}$ is given by
\begin{align}
\dot{\mathcal{V}} & \leq -z_1^T W_1 z_1  + z_1^T W_{12} z_2 - z_2^T W_2 z_2\leq -z^T W z
\end{align}
where $z=[z_1,z_2]^T\in\Re^2$, and the matrices $W_1,W_{12},W_2\in\Re^{2\times 2}$ are defined at \refeqn{W1}-\refeqn{W2}. The matrix $W\in\Re^{2\times 2}$ is given by
\begin{align*}
W=\begin{bmatrix}
\lambda_{m}(W_1) & -\frac{1}{2}\|W_{12}\|_2\\
-\frac{1}{2}\|W_{12}\|_2 & \lambda_m(W_2)
\end{bmatrix}.
\end{align*}
The conditions given at \refeqn{c2}, \refeqn{c1b}, \refeqn{kRkWb} guarantee that all of matrices $M_{11},M_{21},M_{21},M_{22}',W$ are positive definite, and \refeqn{kisigma} implies that the integral term of $\mathcal{V}_1$ at \refeqn{V1} is positive definite about $e_i=\frac{\Delta_x}{k_i}$. This implies that the zero equilibrium of the tracking error is exponentially stable with respect to $e_x,e_v,e_R,e_\Omega$, and the integral terms $e_i,e_I$ are uniformly bounded.

\subsection{Proof of Proposition \ref{prop:Pos2}}\label{sec:pfPos2}

According to the proof of Proposition \ref{prop:Att}, the attitude tracking errors asymptotically decrease to zero, and therefore, they enter the region given by \refeqn{Psi0} in a finite time $t^*$, after which the results of Proposition \ref{prop:Pos} can be applied to yield attractiveness. The remaining part of the proof is showing that the tracking error $z_1=[\|e_x\|,\|e_v\|]^T$ is bounded in $t\in[0, t^*]$. This is similar to the proof given at~\cite{LeeLeo_4457,LeeLeoAJC12}.

\bibliography{ECC13}
\bibliographystyle{IEEEtran}

\end{document}